\documentclass[10pt]{amsart}
\usepackage{amsmath, amsthm, amssymb}
\usepackage[english]{babel}
\usepackage{graphicx}
\usepackage{url}
\usepackage{caption}
\usepackage{float}
\usepackage{tikz-cd}
\usepackage{wrapfig}

\def\wh{\widehat}
\def\wt{\widetilde}
\def\li{\lambda_1}
\def\lii{\lambda_2}
\def\lx{\lambda}

\def\m{{\mathfrak m}}
\def\p{{\mathfrak p}}

\def\lk{{\rm lk}}
\def\delet{\mathaccent"7017 }
\def\K{{L}}
\def\WL{{WL}}
\def\x{{l}}

\def\Z{{\mathbb Z}}
\def\R{{\mathbb R}}

\newtheorem{thm}{Theorem}[section]
\newtheorem{lem}[thm]{Lemma}
\newtheorem{clm}[thm]{Claim}
\newtheorem{cor}[thm]{Corollary}

\newtheorem*{thm*}{Theorem}

\theoremstyle{remark}
\newtheorem{exm}[thm]{Example}
\newtheorem{rem}[thm]{Remark}

\theoremstyle{definition}

\begin{document}

\begin{abstract}
Let $M_1$ and $M_2$ be closed connected orientable $3$-manifolds. We classify the sets of smooth and piecewise linear isotopy classes of embeddings $M_1\sqcup M_2\rightarrow S^6$.
\end{abstract}

\title{The classification of linked $3$-manifolds in $6$-space.}
\author{Sergey Avvakumov}
\thanks{I thank A. Skopenkov, M. Skopenkov, and U. Wagner for useful discussions. Supported in part by RFBR grant 15-01-06302.}
\maketitle

\setcounter{tocdepth}{1}
\tableofcontents

\bigskip
\section{Introduction.}

\smallskip
\subsection{Statement of the result.}
All maps and manifolds in the text are smooth\footnote{In this paper ``smooth'' means $C^1$-smooth. For each $C^\infty$-manifold $N$ the forgetful map from the set of $C^\infty$-isotopy classes of $C^\infty$-embeddings $N\rightarrow \R^m$ to the set of $C^1$-isotopy classes of $C^1$-embeddings $N\rightarrow \R^m$ is a $1$-$1$ correspondence, see \cite{Zh16}, c.f. \cite[footnote 2]{Sk15}.} unless specifically stated otherwise.

For a manifold $N$ denote by $E^m(N)$ the set of isotopy classes of embeddings $N\rightarrow S^m$. 
The main result of the paper is Theorem \ref{thm:main} giving a classification of $E^6(M_1\sqcup M_2)$ for arbitrary closed connected orientable $3$-manifolds $M_1$ and $M_2$. As a corollary we also get a piecewise linear (PL) classification, see Theorem \ref{thm:mainpl} in \S\ref{SS:PL}.

We start with the previously known classifications of $E^6(S^3\sqcup S^3)$ and $E^6(N)$, where $N$ is a closed connected orientable $3$-manifold. These results are later used in our proofs. In \S\ref{SS:survey} we also give a brief general survey on embeddings classification.

An embedding $g:S^3\rightarrow S^6$ is called {\it trivial} if it is isotopic to the standard embedding. The isotopy class of a trivial embedding is also called {\it trivial}. The embedded connected sum operation $\#$ (see \S\ref{SS:sum}) defines a group structure on $E^6(S^3)$. Operation $\#$ also defines an action of $E^6(S^3)$ on $E^6(N)$ for any closed connected orientable $3$-manifold $N$.

\begin{thm}[A. Haefliger]
\label{thm:Hk}
$E^6(S^3)\cong {\mathbb Z}$.
\end{thm}
Let
$$r:E^6(S^3)\rightarrow {\mathbb Z}$$
be one (of the two) isomorphisms $E^6(S^3)\rightarrow {\mathbb Z}$. We call the chosen isomorphism $r$ the {\it Haefliger invariant}\footnote{For arbitrary closed connected orientable $3$-manifold $N$ there is a generalized version $E^6(N)\rightarrow {\mathbb Z}$ of this invariant due to M. Kreck.}.
\begin{rem}
\label{rem:Hk}
The zero of the group $E^6(S^3)$ is the trivial class. I.e., Theorem \ref{thm:Hk} implies that $r(g)=0$ if and only if $g:S^3\rightarrow S^6$ is trivial.
\end{rem}

All the homology groups in the text are with coefficients in $\mathbb Z$ unless another group is explicitly specified. For any closed connected orientable $3$-manifold $N$ the {\it Whitney invariant} $$W:E^6(N)\rightarrow H_1(N)$$ is defined in \cite{Sk08}. We give an equivalent definition in \S\ref{SS:whitney}.

For an element $a\neq 0$ of a free abelian group $G$ denote by ${\rm div}(a)$ the {\it divisibility} of $a$. I.e., ${\rm div}(a)$ is the maximal positive integer such that $a={\rm div}(a)b$ for some $b\in G$. Put ${\rm div}(0)=0$. For an element $a$ of an abelian group $G$ denote by ${\rm div}(a)$ the divisibility of the projection of $a$ to the free part of $G$.

\begin{thm}[A. Skopenkov, others]\footnote{Part (III) of the Theorem is due to A. Skopenkov, see \cite{Sk08}. Parts (I) and (II) were known earlier, see \cite[Footnote 3]{Sk08}.}
\label{thm:Sk}
For any closed connected orientable $3$-manifold $N$
\begin{itemize}
\item [{\bf (I)}] the Whitney invariant $$W:E^6(N)\rightarrow H_1(N)$$ is surjective.
\item [{\bf (II)}] The embedded connected sum action of $E^6(S^3)$ is transitive on each of the preimages of $W$.
\item [{\bf (III)}] For any $[f]\in E^6(N)$ and $[g]\in E^6(S^3)$ we have that $[f]\#[g]=[f]$ if and only if the Haefliger invariant $r(g)$ is a multiple of the divisibility of the Whitney invariant $W(f)$, i.e., $r(g)=k{\rm div}(W(f))$ for some integer $k$.
\end{itemize}
\end{thm}

\begin{cor}
\label{cor:Sk}
Suppose that $H_1(N)$ is infinite. Then there is an element $[f]\in E^6(N)$ and a non-trivial element $[g]\in E^6(S^3)$ such that $[f]\#[g]=[f]$.
\end{cor}

An embedding $g:S^3\sqcup S^3\rightarrow S^6$ is called {\it unlinked} if its components lie in pairwise disjoint balls. An unlinked embedding $g:S^3\sqcup S^3\rightarrow S^6$ is called {\it trivial} if its restriction to each component is trivial. The isotopy class of a trivial (resp. unlinked) embedding is also called {\it trivial} (resp. {\it unlinked}). An unlinked embedding differs from a trivial embedding only by the ``knotting'' of the components. The component-wise embedded connected sum operation $\#$ (see \S\ref{SS:sum}) defines a group structure on $E^6(S^3\sqcup S^3)$ and an action of $E^6(S^3\sqcup S^3)$ on $E^6(M_1\sqcup M_2)$ for arbitrary closed connected orientable $3$-manifolds $M_1$ and $M_2$.

For $k\in\{1,2\}$ let
$$r_k:E^6(S^3\sqcup S^3)\rightarrow \Z$$
be the Haefliger invariant of the restriction to the $k$-th connected component.
The (defined later in \S\ref{SS:lambda}) isotopy invariants 
$$\li,\lii:E^6(S^3\sqcup S^3)\rightarrow{\mathbb Z}$$
are called the (generalized) {\it linking coefficients}.

Denote $$\wt{\Z^4}:=\{(a,b)\in {\mathbb Z}^2|a\equiv b\pmod{2}\}\times {\mathbb Z}^2\subset \Z^4.$$

\begin{thm}[A. Haefliger, \cite{Ha62}]
\label{thm:H}
The map $\li\times\lii\times r_1\times r_2:E^6(S^3\sqcup S^3)\rightarrow \Z^4$ is a monomorphism and its image is $\wt{\Z^4}$.
\end{thm}

\begin{rem}
\label{rem:Hk}
The zero of the group $E^6(S^3\sqcup S^3)$ is the trivial class. I.e., Theorem \ref{thm:H} implies that $r_1(g)=r_2(g)=\li(g)=\lii(g)=0$ if and only if $g:S^3\sqcup S^3\rightarrow S^6$ is trivial. Also, $\li(g)=\lii(g)=0$ if and only if $g:S^3\sqcup S^3\rightarrow S^6$ is unlinked; the ``if'' part follows from the definitions of $\li$ and $\lii$, and the ``only if'' part follows from the PL version of Theorem \ref{thm:H}, see Theorem \ref{thm:Hpl}.
\end{rem}

We use Theorems \ref{thm:Hk}, \ref{thm:Sk}, \ref{thm:H} to prove Theorem \ref{thm:main} which is the main result of the paper. 
First we present two corollaries of Theorem \ref{thm:main} showing that the connection between Theorems \ref{thm:Hk}, \ref{thm:H}, \ref{thm:Sk} on one hand and Theorem \ref{thm:main} on the other hand is not trivial. The corollaries are proved at the end of this subsection.

For the rest of the text let $M_1$ and $M_2$ be some closed connected orientable $3$-manifolds.

\begin{cor}
\label{cor:1}
Suppose that $H_1(M_1)$ is infinite. Then there is an element $[f]\in E^6(M_1\sqcup M_2)$ and a non-trivial not unlinked element $[g]\in E^6(S^3\sqcup S^3)$ such that $[f]\#[g]=[f]$.
\end{cor}

\begin{rem}
If one omits the ``$g$ is not unlinked'' part of the statement, the corollary above will trivially follow from Theorem \ref{thm:Sk} (cf. Corollary \ref{cor:Sk}).
\end{rem}

\begin{cor}
\label{cor:2}
There are manifolds $M_1$, $M_2$, an element $[f]\in E^6(M_1\sqcup M_2)$, and an unlinked element $[g]\in E^6(S^3\sqcup S^3)$, such that the restrictions of $[f]$ and $[f]\#[g]$ to each connected component are isotopic, but $[f]\neq[f]\#[g]$.
\end{cor}

\begin{rem}
Informally, Corollary \ref{cor:Sk} means that we can sometimes unknot spherical knots by ``dragging'' them along a knotted manifold $M_1$ with infinite $H_1(M_1)$. Corollary \ref{cor:2} then means that sometimes this procedure is made impossible by the presence of another manifold $M_2$ linked with $M_1$.
\end{rem}

For an embedding $f:M_1\sqcup M_2\rightarrow S^6$ and $k\in\{1,2\}$ define
$$W_k:E^6(M_1\sqcup M_2)\rightarrow H_1(M_k) \text{\quad by the formula\quad} W_k(f)=W(f|_{M_k}).$$
I.e., $W_k(f)$ is the Whitney invariant of the restriction of $f$ to the $k$-th connected component.
The map
$$\K_1\times \K_2:E^6(M_1\sqcup M_2)\rightarrow H_1(M_1)\times H_1(M_2)$$
is defined below in \S\ref{SS:whitney}. All four $W_1$, $\K_1$, $W_2$, $\K_2$ are called (generalized) Whitney invariants. 

For brevity we denote 
$$\WL:=W_1\times \K_1\times W_2\times \K_2$$
for the rest of the text.

For any $[f]\in E^6(M_1\sqcup M_2)$ let ${\rm Stab}_f\subset \wt{\Z^4}$ be the subgroup generated by all elements
\begin{itemize}
\item $(0,2\K_1f\cap \alpha,W_1f\cap \alpha, 0)\in \wt{\Z^4},$
\item $(2\K_1f\cap \beta,2W_1f\cap \beta,0, 0)\in \wt{\Z^4},$
\item $(2\K_2f\cap \gamma,0,0,W_2f\cap \gamma)\in \wt{\Z^4},$
\item $(2W_2f\cap \delta,2\K_2f\cap \delta,0, 0)\in \wt{\Z^4},$
\end{itemize}
where $\alpha$, $\beta$ take all values in $H_2(M_1)$ and $\gamma$, $\delta$ take all values in $H_2(M_2)$, and $\cap$ denotes the cap product.

\begin{thm}
\label{thm:main}
For any closed connected orientable $3$-manifold $M_1$ and $M_2$
\begin{itemize}
\item [{\bf (I)}] the map
$$\WL:E^6(M_1\sqcup M_2)\rightarrow H_1(M_1)\times H_1(M_1)\times H_1(M_2)\times H_1(M_2)$$
is surjective.

\item [{\bf (II)}] The embedded connected sum action of $E^6(S^3\sqcup S^3)$ is transitive on each of the preimages of $\WL$.

\item [{\bf (III)}] For any $[f]\in E^6(M_1\sqcup M_2)$ and $[g]\in E^6(S^3\sqcup S^3)$ the class $[g]$ is in the stabilizer of $[f]$ under the action $\#$ if and only if $$(\li\times\lii\times r_1\times r_2)(g)\in {\rm Stab}_f\subset\wt{\Z^4}.$$
\end{itemize}
\end{thm}

Parts (II) and (III) of Theorem \ref{thm:main} can be restated in terms of description the preimages of $\WL$.

\begin{thm}
\label{thm:mainc}
For any $[f]\in E^6(M_1\sqcup M_2)$ there is a surjective map
$$\phi_{[f]}:\wt{\Z^4}\rightarrow \WL^{-1}\WL(f)$$
such that for any $x,y\in \wt{\Z^4}$ we have $\phi_{[f]}(x)=\phi_{[f]}(y)$ if and only if $x-y\in {\rm Stab}_f$.
\end{thm}

\begin{rem}
In the prequel \cite{Av16} the author proved Theorem \ref{thm:main} in the special case of $M_1=S^1\times S^2$, $M_2=S^3$ and only for embeddings $S^1\times S^2\sqcup S^3\rightarrow S^6$ whose restrictions to $S^1\times S^2$ and $S^3$ are isotopic to the standard embeddings. Unfortunately, methods used there do not work in the general case. For instance, Corollary \ref{cor:2} cannot be deduced from \cite{Av16}.
\end{rem}

\begin{exm}
Suppose that $M_1$ and $M_2$ are homology spheres. Then the action $\#$ is transitive and free, and thus gives a $1$-$1$ correspondence between $E^6(M_1\sqcup M_2)$ and $E^6(S^3\sqcup S^3)$.
\end{exm}

\begin{exm}
Suppose that $M_1$ and $M_2$ are rational homology spheres (for instance $M_1=M_2={\mathbb R}P^3$). Then each of $|H_1(M_1)|\cdot|H_1(M_2)|$ preimages of $\WL$ is in $1$-$1$ correspondence with $E^6(S^3\sqcup S^3)$.
\end{exm}

\begin{proof}[Proof of Corollary \ref{cor:1}]
Since $H_1(M_1)$ is infinite, there are $\alpha'\in H_1(M_1)$ and $\alpha\in H_2(M_1)$ such that $\alpha'\cap\alpha=1$.

By part (I) of Theorem \ref{thm:main}, there is an embedding $f:M_1\sqcup M_2\rightarrow S^6$ such that $W_1f=0$ and $\K_1f=\alpha'$.

By Theorem \ref{thm:H}, there is an embedding $g:S^3\sqcup S^3\rightarrow S^6$ such that $$(\li\times\lii\times r_1\times r_2)(g)=(0,2,0,0)=(0,2\K_1f\cap \alpha,W_1f\cap \alpha, 0)\in{\rm Stab}_f.$$

Embeddings $f$ and $g$ are as required. Indeed, $g$ is not unlinked, see Remark \ref{rem:Hk}, and $[f]=[f]\#[g]$ by part (III) of Theorem \ref{thm:main}.
\end{proof}

\begin{proof}[Proof of Corollary \ref{cor:2}]
Take $M_1=S^1\times S^2$ and $M_2=S^3$. By part (I) of Theorem \ref{thm:main}, there is an embedding $f:S^1\times S^2\sqcup S^3\rightarrow S^6$ such that $W_1(f)=L_1(f)=[S^1\times *]$ and $W_2(f)=L_2(f)=0$.

By Theorem \ref{thm:H}, there is an embedding $g:S^3\sqcup S^3\rightarrow S^6$ such that $(\li\times\lii\times r_1\times r_2)(g)=(0,0,1,0)$.

Let us prove that $f$ and $g$ are as required. Embedding $g$ is unlinked, see Remark \ref{rem:Hk}. By part (III) of Theorem \ref{thm:Sk}, we have that the restrictions of $[f]$ and $[f]\#[g]$ to each connected component are isotopic.

It remains to check that $[f]\neq[f]\#[g]$. The group ${\rm Stab}_f$ is generated by two elements, $(0,2,1,0)$ and $(2,2,0,0)$ of $\wt{\Z^4}$ (one can obtain these generators by substituting $\alpha=\beta=[*\times S^2]$ in the definition of ${\rm Stab}_f$). Clearly, $(\li\times\lii\times r_1\times r_2)(g)=(0,0,1,0)$ is not a linear combination of $(0,2,1,0)$ and $(2,2,0,0)$. So, $[f]\neq[f]\#[g]$ by part (III) of Theorem \ref{thm:main}.
\end{proof}

\smallskip
\subsection{PL version of the main result.}
\label{SS:PL}
For a PL manifold $N$ denote by $E^m_{PL}(N)$ the set of PL isotopy classes of PL embeddings $N\rightarrow S^m$.

In this subsection $M_k$ also denotes the PL manifold obtained by triangulating the smooth manifold $M_k$. In dimension $3$ any PL manifold may be obtained in this way, see for example \cite{Wh}.

The definition of linking coefficients
$$\li,\lii:E^6_{PL}(S^3\sqcup S^3)\rightarrow{\mathbb Z},$$
of Whitney invariants
$$\WL:E^6_{PL}(M_1\sqcup M_2)\rightarrow H_1(M_1)\times H_1(M_1)\times H_1(M_2)\times H_1(M_2),$$
and of the (componentwise) embedded connected sum $\#$ carries over from the smooth category without any changes.

The set $E^6_{PL}(S^3\sqcup S^3)$ is still a group with $\#$ being the group operation.

\begin{thm}[A. Haefliger, \cite{Ha62}]
\label{thm:Hpl}
The map $\li\times\lii:E^6_{PL}(S^3\sqcup S^3)\rightarrow \Z^2$ is a monomorphism and its image is $\{(a,b)\in {\mathbb Z}^2|a\equiv b\pmod{2}\}$.
\end{thm}

For any $[f]\in E^6_{PL}(M_1\sqcup M_2)$ let ${\rm Stab}_{PL,f}\subset \Z^2$ be the subgroup generated by all elements
\begin{itemize}
\item $(0,2\K_1f\cap \alpha),$
\item $(2\K_1f\cap \beta,2W_1f\cap \beta),$
\item $(2\K_2f\cap \gamma,0),$
\item $(2W_2f\cap \delta,2\K_2f\cap \delta),$
\end{itemize}
where $\alpha$, $\beta$ take all values in $H_2(M_1)$ and $\gamma$, $\delta$ take all values in $H_2(M_2)$.

\begin{rem}
In the definition of ${\rm Stab}_{PL,f}$ one can replace $(0,2\K_1f\cap \alpha)$ and $(2\K_2f\cap \gamma,0)$ by $(0, 2{\rm div}(\K_1f))$ and $(2{\rm div}(\K_2f),0)$, respectively. We do not know of any further simplifications.
\end{rem}

\begin{thm}
\label{thm:mainpl}
For any closed connected orientable $PL$ $3$-manifold $M_1$ and $M_2$
\begin{itemize}
\item [{\bf (I)}] the map
$$\WL:E^6_{PL}(M_1\sqcup M_2)\rightarrow H_1(M_1)\times H_1(M_1)\times H_1(M_2)\times H_1(M_2)$$
is surjective.

\item [{\bf (II)}] The embedded connected sum action of $E^6_{PL}(S^3\sqcup S^3)$ is transitive on each of the preimages of $\WL$.

\item [{\bf (III)}] For any $[f]\in E^6_{PL}(M_1\sqcup M_2)$ and $[g]\in E^6_{PL}(S^3\sqcup S^3)$ the class $[g]$ is in the stabilizer of $[f]$ under the action $\#$ if and only if $(\li\times\lii)(g)\in {\rm Stab}_{PL,f}\subset \Z^2.$
\end{itemize}
\end{thm}

\smallskip
\subsection{A very brief survey on embeddings classification.}
\label{SS:survey}
According to E. C. Zeeman (\cite{Z}, \cite{MA_emb}), three major classical problems of topology are the following.
\begin{itemize}
\item {\it Homeomorphism Problem:} Classify $n$-manifolds.
\item {\it Embedding Problem:} Find the least dimension $m$ such that given space admits an embedding into $m$-dimensional Euclidean space ${\mathbb R}^m$.
\item {\it Knotting Problem:} Classify embeddings of a given space into another given space up to isotopy.
\end{itemize}
This paper is on a special case of the third problem.

Let us start with the known results on the sets $E^m(S^n)$ and $E^m_{PL}(S^n)$.
The set $E^3(S^1)$ (or $E^3_{PL}(S^1)$) is studied in the classical knot theory which produced a lot of beautiful results in the last $200$ years. However, relatively early it was understood that a complete classification of $E^3(S^1)$ is probably unachievable. In general, there is no known complete classification of $E^m(S^n)$ for $m=n+2\geq 3$.

The situation is much better when $m\geq n+3$ (codimension at least $3$ case). So, until the end of this subsection we assume that $m\geq n+3$. 

The following two theorems establish that there are no knots in case when the codimension $m-n$ is large enough. Somewhat surprisingly, the precise meaning of ``large enough'' is different in the smooth and PL categories.

\begin{thm*}[E. C. Zeeman, {\cite[Theorem $2$]{Z63}}]
$|E^m_{PL}(S^n)|=1.$
\end{thm*}

\begin{thm*}[A. Haefliger]
If $2m\geq 3n+4$ then $|E^m(S^n)|=1.$
\end{thm*}

As it was said earlier, the sets $E^m(S^n)$ and $E^m_{PL}(S^n)$ have a group structure given by the embedded connected sum operation. The following is a generalisation of Theorem \ref{thm:Hk}.

\begin{thm*}[A. Haefliger]
$E^{3k}(S^{2k-1})\cong \Z$ for $k>0$ even and $E^{3k}(S^{2k-1})\cong \Z_2$ for $k>1$ odd.
\end{thm*}

There is a special embedding $S^{2k-1}\rightarrow S^{3k}$, also called the {\it Haefliger trefoil knot} (see \cite{Ha62sp}), which is a generator of $E^{3k}(S^{2k-1})\cong \Z$ for even $k$. It is not known, however, if the Haefliger trefoil knot is the generator of $E^{3k}(S^{2k-1})\cong \Z_2$ for odd $k$, see \cite{MA_knots}.

Let us now mention a few results on the knotting of manifolds different from spheres. 

For any $n$-connected PL $m$-manifold $N$ (recall that $m\geq n+3$) the set $E^{2m-n}_{PL}(N)$ was classified by J. Vrabec in \cite{V}.

For any smooth connected $4$-manifold $N$ the set $E^7(N)$ was classified only recently and only up to the embedded connected sum action of $E^7(S^4)$ by D. Crowley and A. Skopenkov in \cite{Cr16}. In the sequel \cite{Cr16_1} the authors strengthened this result. In the special case $H_1(N)=0$ a complete classification of $E^7(N)$ was obtained much earlier by J. Bo\'{e}chat and A. Haefliger in \cite{B70}. See also \cite{B70_1} for the generalisation to the case of $E^{6k+1}(N)$, where $N$ is $4k$-dimensional.

Finally, let us get back to links, i.e., isotopy classes of embeddings of manifolds with several connected components.
Denote by $S^n_{(k)}$ the disjoint union of $k$ copies of $S^n$.

\begin{thm*}[A. Haefliger, \cite{Ha66}]
There is an isomorphism
$$E^m(S^n_{(k)})\rightarrow E^m_{PL}(S^n_{(k)})\oplus\underset{i=1}{\overset{k}{\bigoplus}}E^m(S^n).$$
Composition of the isomorphism with the projection to $E^m_{PL}(S^n_{(k)})$ is the forgetful map.
Composition of the isomorphism with the projection to the $i$-th summand of $\bigoplus E^m(S^n)$ is the isotopy class of the $i$-th connected component.
\end{thm*}

In other words, in codimension at least $3$ smooth and PL links of spheres differ only by smooth knotting of each connected component.

For $1\leq i,j\leq k$, $i\neq j$, let 
$$\lambda_{ij}:E^m_{PL}(S^n_{(k)})\rightarrow \pi_n(S^{m-n-1})$$
be the (generalized) {\it linking coefficient} of the $i$-th and the $j$-th connected components, i.e., the homotopy class of $i$-th component in the compliment to the $j$-th component. The map $\lambda_{ij}$ is analogously defined in the smooth category  (in the special case $k=2$, $n=3$, $m=6$, we denote $\lambda_{12}$ and $\lambda_{21}$ simply as $\lambda_1$ and $\lambda_2$, respectively, throughout the rest of the paper).

\begin{thm*}[A. Haefliger, \cite{Ha66}]
The collection of pairwise linking coefficients is bijective for $2m\geq3n+4$ and $E^m_{PL}(S^n_{(k)})$.
\end{thm*}

\begin{thm*}[A. Haefliger, \cite{Ha62}]
When $k\geq2$, $k\neq 3,7$ the homomorphism
$$\lambda_{12}\oplus\lambda_{21}:E^{3k}_{PL}(S^{2k-1}\sqcup S^{2k-1})\rightarrow \pi_{2k-1}(S^k)\oplus \pi_{2k-1}(S^k)$$
is injective and its image is $\{(a,b):\Sigma a=\Sigma b\}$.
\end{thm*}
Combining this theorem ($k=2$) with some of the other theorems above one gets Theorem \ref{thm:H}, i.e., a classification of $E^6(S^3\sqcup S^3)$.

All the results we mentioned so far were either
\begin{itemize}
\item in the metastable range $2m\geq 3n+4$,
\item or on links of (homology) spheres,
\item or on connected manifolds.
\end{itemize}
Therefore, Theorem \ref{thm:main} is the first embeddings classification result (that we are aware of) which falls into none of those three categories.

\smallskip
\subsection{Definition of embedded connected sum $\#$.}
\label{SS:sum}
Let $f:M_1\rightarrow S^6$ and $g:S^3\rightarrow S^6$ be embeddings. Take representatives $f'\in [f]$ and $g'\in [g]$ such that the images of $f'$ and $g'$ lie in disjoint balls. Connect the images of $f'$ and $g'$ by a thin tube along an arc. The isotopy class of the obtained embedding is called an {\it embedded connected sum} of $f$ and $g$ and is denoted by $[f]\#[g]$. The independence on the choice of the representatives, the arc, and the tube follows by an argument analogous to \cite[Standardization Lemma, case $(p,q,m)=(0,3,6)$]{Sk15}.

For embeddings $f:M_1\sqcup M_2\rightarrow S^6$ and $g:S^3\sqcup S^3\rightarrow S^6$ their {\it component-wise} embedded connected sum is defined analogously and is also denoted by $[f]\#[g]$, see Fig.\ref{f:_connected_sum}.

	\begin{figure}[H]
	\begin{center}
	\includegraphics{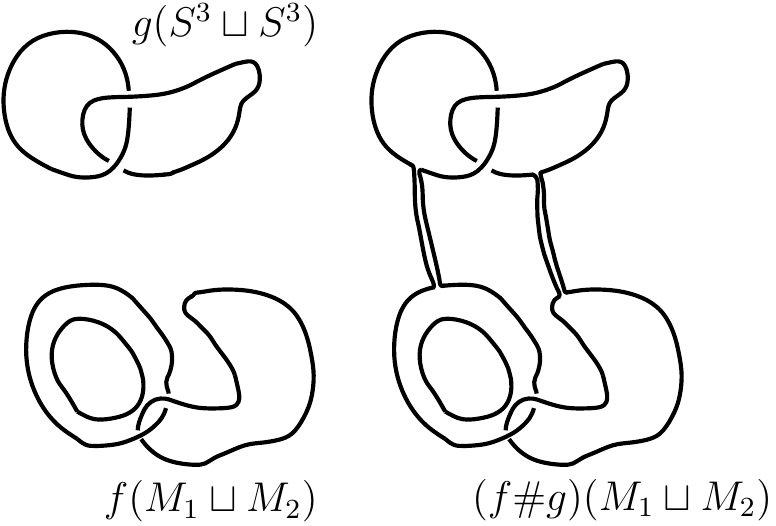}
	\caption{The componentwise embedded connected sum $\#$.}
	\label{f:_connected_sum}
	\end{center}
	\end{figure}

The described operation $\#$ defines a group structure on $E^6(S^3)$ (or $E^6(S^3\sqcup S^3)$) and an action of $E^6(S^3)$ (or $E^6(S^3\sqcup S^3)$) on $E^6(M_1)$ (or $E^6(M_1\sqcup M_2)$).

\smallskip
\subsection{Definition of linked embedded connected sum $\#_1$, $\#_2$.}
\label{SS:lk_sum}
Let $f:M_1\sqcup M_2\rightarrow S^6$ and $g:S^3\rightarrow S^6$ be embeddings with disjoint images. For $k\in\{1,2\}$ connect $f(M_k)$ with $g(S^3)$ by a thin tube along an arc. Denote the obtained embedding $M_1\sqcup M_2\rightarrow S^6$ by $f\#_k g$. It is called a {\it linked embedded connect sum} of $f$ and $g$. Clearly, the embedding $f\#_k g$ depends on the choice of the arc and the tube, but we drop them from the notation. See Fig.\ref{f:_linked_connected_sum}.

	\begin{figure}[H]
	\begin{center}
	\includegraphics{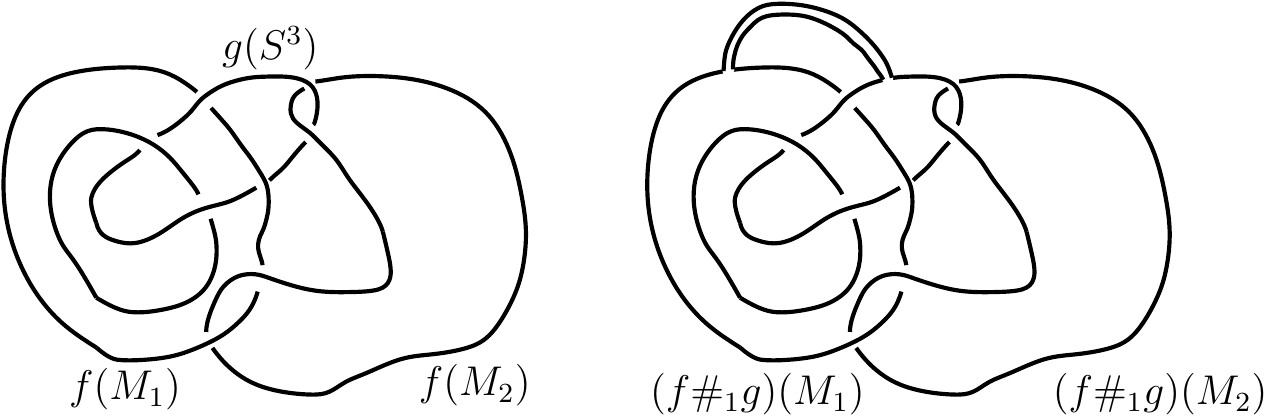}
	\caption{The linked embedded connected sum $\#_1$.}
	\label{f:_linked_connected_sum}
	\end{center}
	\end{figure}

For the fixed embeddings $f$ and $g$ the isotopy class $[f\#_k g]$ is well defined, i.e., it does not depend on the choice of the arc or the tube.
This can be proved analogously to \cite[Standardization Lemma, case $(p,q,m)=(0,3,6)$]{Sk15} (the independence on the choice of the arc also easily follows from the fact that the images of $f$ and $g$ have codimension greater than $2$).

\smallskip
\subsection{Definition of the Whitney invariants $W$ and $\K_k$.}
\label{SS:whitney}
Let $N$ be a closed connected orientable $3$-manifold. Our definition of the Whitney invariant $W:E^6(N)\rightarrow H_1(N)$ is equivalent to the one given in \cite{Sk08}.

Let $f,f':N\rightarrow S^6$ be embeddings. Consider a general position homotopy $F:N\times I\rightarrow S^6\times I$ between $f$ and $f'$. The Whitney invariant of the pair $(f,f')$ is the homology class
$$W(f,f'):=[\{x\in N\times I:|F^{-1}Fx|\geq 2 \}]\in H_1(N\times I)=H_1(N)$$
which can be defined as in \cite{SkSurvey}.

To define $W$ for a single embedding (as opposed to a pair $(f,f')$ of embeddings) we need to choose some ``base embedding''. Manifold $N$ is orientable, so it embeds into $S^5$, see \cite{Hi61}. Let $f^0_N:N\rightarrow S^6$ be an embedding with the image in $S^5\subset S^6$. For any $f:N\rightarrow S^6$ denote $$W(f):=W(f^0_N,f).$$

We choose $f^0_{M_1}$ and $f^0_{M_2}$ so that their images lie in disjoint $6$-balls. Define $$f^0:M_1\sqcup M_2\rightarrow S^6\text{\quad by the formula\quad} f^0=f^0_{M_1}\sqcup f^0_{M_2}.$$

Recall that for $k\in\{1,2\}$ and for an embedding $f:M_1\sqcup M_2\rightarrow S^6$ we earlier defined
$$W_k:E^6(M_1\sqcup M_2)\rightarrow H_1(M_k) \text{\quad by the formula\quad} W_k(f)=W(f|_{M_k}).$$

Let us now define $\K_1$ and $\K_2$. Let $f,f':M_1\sqcup M_2\rightarrow S^6$ be embeddings. Consider a general position homotopy $F:(M_1\sqcup M_2)\times I\rightarrow S^6\times I$ between $f$ and $f'$. The Whitney invariants $\K_1$ and $\K_2$ of the pair $(f,f')$ are the homology classes
\begin{itemize}
\item[] $\K_1(f,f'):=[(F|_{M_1\times I})^{-1}(F(M_1\times I)\cap F(M_2\times I))]\in H_1(M_1),$
\item[] $\K_2(f,f'):=[(F|_{M_2\times I})^{-1}(F(M_1\times I)\cap F(M_2\times I))]\in H_1(M_2).$
\end{itemize}
For $k\in\{1,2\}$ and any $f:M_1\sqcup M_2\rightarrow S^6$ denote
$$\K_k(f):=\K_k({f^0}, f)\in H_1(M_k).$$

\smallskip
\subsection{Proof of part {\bf (I)} of Theorem \ref{thm:main}.}
The following claim is essentially proved (but not explicitly stated) in \cite[``Construction of an arbitrary embedding $N\rightarrow \R^6$ from a fixed embedding $g:N\rightarrow \R^5$'']{Sk08}. For the readers convenience we present (a very similar) proof here. In the proof an later in the text we use the standard notation $V_{m,n}$ for the Stiefel manifold of $n$-frames in $\R^m$. All the framings (resp. frames) in the text are {\it normal} framings (resp. frames) compatible with orientation (in the case of framings).

\begin{clm}
\label{clm:wsurj}
Let $f:M_1\sqcup M_2\rightarrow S^6$ be an embedding and $a\in H_1(M_1)$ a homology class. Then there is an embedding $g:D^4\rightarrow S^6$ such that 
\begin{itemize}
\item $g(S^3)\cap {\rm Im}(f)=\emptyset,$
\item $g(D^4)\cap f(M_2)=\emptyset,$
\item $[(f|_{M_1})^{-1}g(D^4)]=a.$
\end{itemize}
\end{clm}
\begin{proof}
Represent $a$ by an oriented circle in $M_1$ and denote the circle by the same letter. Consider a normal framing $\alpha$ of $f(a)$ in $f(M_1)$. Extend it to a normal framing $\alpha,\beta$ of $f(a)$ in $S^6$, where $\beta$ is normal to $f(M_1)$. The extension exists because $f(a)$ is unknotted in $S^6$ and so the obstruction to the existence of the extension is in $\pi_1(V_{5,2})=0$.

By general position there is a $2$-disk $\Delta$ in $S^6$ such that
\begin{itemize}
\item $\partial \Delta = f(a)$,
\item ${\rm Int}\Delta\cap f(M_1\sqcup M_2)=\emptyset,$
\item the first vectors of $\beta$ ``looks'' inside of $\Delta$. 
\end{itemize} 

Denote by $\beta'$ the normal $2$-frame of $f(a)$ made out of the last two vectors of $\beta$. Extend $\beta'$ to a normal $2$-frame of $\Delta$. The extension exists because the obstruction to its existence is in $\pi_1(V_{4,2})=0$. The vectors of $\beta'$ on $\Delta$ plus the vectors of $\beta$ on $\partial \Delta=f(a)$ give us an embedding $g:D^4\rightarrow D^6$ which is as required.
\end{proof}

\begin{proof}[Proof of part {\bf (I)} of Theorem \ref{thm:main}]
We need to prove that $\WL$ is surjective.

Take any element $a'\in H_1(M_1)$ and any embedding $f:M_1\sqcup M_2\rightarrow S^6$. Denote $a:=a'-W_1(f)$. Let $g:D^4\rightarrow S^6$ be an embedding given by Claim \ref{clm:wsurj}.

Consider the embedding $f':=f\#_1(g|_{S^3})$. There is a homotopy between $f'$ and $f$ contracting $g(S^3)$ along the disk $g(D^4)$. By the definition of the Whitney invariants and by the construction of $g$, we have $W_1(f')=W_1(f)+a=a'$, and $W_2(f')=W_2(f)$, $\K_1(f')=\K_1(f)$, $\K_2(f')=\K_2(f)$. So, we can change the value of the Whitney invariant $W_1$ of an embedding to any desired value $a'$ without changing the other three Whitney invariants.

Similarly to the previous paragraph (take $f'':=f\#_2(g|_{S^3})$ instead of $f'$) we can change the value of the Whitney invariant $L_1$ of an embedding to any desired value $a'$ without changing the other three Whitney invariants.

Similarly to previous two paragraphs we can also change $W_2$ and $\K_2$ in the same manner. So, $\WL$ is surjective, because there exists at least one embedding $M_1\sqcup M_2\rightarrow S^6$ (for instance take $f^0$).
\end{proof}

\smallskip
\subsection{Definition of the linking coefficients $\li$ and $\lii$ and their relation to the Haefliger invariant $r$.}
\label{SS:lambda}
Let $g:S^3_1\sqcup S^3_2\rightarrow S^6$ be an embedding, where $S^3_1$ and $S^3_2$ are two copies of $S^3$. Choose an oriented disk $D^3_g\subset S^6$ intersecting $g(S^3_2)$ transversally at a single point of positive sign. Identify $H_2(S^6\setminus gS^3_2)$ with $\mathbb Z$ by identifying $[\partial D^3_g]\in H_2(S^6\setminus gS^3_2)$ with $1\in{\mathbb Z}$. Identify $H_2(S^2)$ with $\Z$ by choosing an orientation of $S^2$. Choose a homotopy equivalence $h:S^6\setminus gS^3_2\rightarrow S^2$ which induces the identity isomorphism in $H_2$. Define the first linking coefficient by the formula
$$\li(g):=[hg|_{S^3_1}]\in \pi_3(S^2)={\mathbb Z},$$
where identification $\pi_3(S^2)={\mathbb Z}$ identifies the homotopy class of the Hopf map with $1$. All the orientation preserving homotopy equivalences $S^2\rightarrow S^2$ are homotopic to each other, so $\li$ is well-defined.

The definition of the second linking coefficient $\lii$ is analogous and is obtained by the exchange of the components,
$$\lii(g):=\li(g'),$$
where $g':S^3_1\sqcup S^3_2\rightarrow S^6$ is such that $g'|_{S^3_2}=g|_{S^3_1}$ and $g'|_{S^3_1}=g|_{S^3_2}$.

Let $A,B:S^3\rightarrow S^6$ be embeddings with disjoint images. For brevity denote
$$\lx(A,B):=\li(A\sqcup B).$$
Informally, $\lx(A,B)$ is the homotopy class of $A$ in the compliment to $B(S^3)$.

The following lemma easily follows from the definition of $\lx$.
\begin{lem}
\label{lem:Lass}
Let $A,B,C:S^3\rightarrow S^6$ be embeddings with pairwise disjoint images. Then 
$$\lx(A\#B,C)=\lx(A,C)+\lx(A,B).$$
\end{lem}

\begin{rem}
Note that $\lx(A,B\# C)$ is not necessarily equal to $\lx(A,B)+\lx(A,C)$ even if the images of $B$ and $C$ lie in pairwise disjoint $6$-balls. As an example one can take Borromean rings $A,B,C:S^3\rightarrow S^6$. Then $A\sqcup B\#C:S^3\sqcup S^3\rightarrow S^6$ is the Whitehead link with $\lx(A,B\# C)=2\neq 0+0=\lx(A,B)+\lx(A,C)$, see \cite[Lemma 2.18]{Sk15}.
\end{rem}

For the proof of the following lemma see \cite[Lemma 2.16]{Sk15}.
\begin{lem}
\label{lem:conn}
Let $A,B:S^3\rightarrow S^6$ be embeddings with disjoint images. Then 
$$r(A\#B)=r(A)+r(B)+\frac{\lx(A,B)+\lx(B,A)}{2}.$$
\end{lem}

In particular, $r(A\#B)=r(A)+r(B)$ if $A(S^3)$ and $B(S^3)$ lie in disjoint $6$-balls.

\begin{rem}
The number $\frac{\lx(A,B)+\lx(B,A)}{2}$ is integer by Haefliger Theorem \ref{thm:H}.
\end{rem}

\smallskip
\subsection{Proof of ``PL'' Theorem \ref{thm:mainpl} modulo ``smooth'' Theorem \ref{thm:main}.}
For a piecewise smooth (PS) manifold $N$ denote by $E^m_{PS}(N)$ the set of PS isotopy classes of PS embeddings $N\rightarrow S^m$. The forgetful map $E^m_{PL}(N)\rightarrow E^m_{PS}(N)$ is a bijection, see \cite[\S 2.2]{Ha67}. Therefore, Theorem \ref{thm:mainpl} can be restated in the PS category without any changes. For our convenience we shall prove the PS version of Theorem \ref{thm:mainpl}.

Let $${\rm Fg}:E^6(M_1\sqcup M_2)\rightarrow E^6_{PS}(M_1\sqcup M_2)$$ be the {\it forgetful} map.

\begin{lem}
\label{lem:fg}
The forgetful map ${\rm Fg}$ has the following properties.
\begin{itemize}
\item [(1)] ${\rm Fg}$ preserves the invariants $\li$, $\lii$, and $\WL$.
\item [(2)] ${\rm Fg}$ commutes with $\#$, i.e., ${\rm Fg}([f]\#[g])={\rm Fg}([f])\#{\rm Fg}([g])$ for any $[f]\in E^6(M_1\sqcup M_2)$ and $[g]\in E^6_{PL}(S^3\sqcup S^3)$.
\item [(3)] ${\rm Fg}$ is surjective.
\item [(4)] Suppose that ${\rm Fg}([f'])={\rm Fg}([f])$ for some $[f], [f']\in E^6(M_1\sqcup M_2)$. Then there is $[g]\in E^6(S^3\sqcup S^3)$ such that $[f']=[f]\#[g]$ and that $[g]$ is unlinked, i.e., $\li(g)=\lii(g)=0$.
\end{itemize}
\end{lem}
\begin{proof}
(1), (2) follow by the definitions of $\li$, $\lii$, $\WL$, and $\#$.

Let us prove (3). The obstruction to smoothing any PS embedding $M_1\sqcup M_2\rightarrow S^6$ lies in groups $H^{i+1}(M_1\sqcup M_2; E^{i+3}(S^i))$ for $i=0,1,2$, see \cite[First paragraph of introduction]{B70_1}, \cite[Proof of Lemma 7]{Hu72}. Since $E^3(S^0)=E^4(S^1)=E^5(S^2)=0$, the obstruction vanishes.

It remains to prove (4). Let $F:(M_1\sqcup M_2)\times I\rightarrow S^6\times I$ be a PS isotopy between $f$ and $f'$. The only obstruction to smoothing $F$ is some cohomology class $a\in H^4((M_1\sqcup M_2)\times I; E^6(S^3))\cong E^6(S^3)\oplus E^6(S^3)$. Choose an unlinked embedding $g:S^3\sqcup S^3\rightarrow S^6\times 0$ whose image is in a $6$-ball disjoint with the image of $F_0$ and such that $r_1(g)\oplus r_2(g)=a$. A PS embedding $G:D^4\sqcup D^4\rightarrow S^6\times I$ is obtained from $g$ by coning over two generic points. Then $F\# G$ is a PS concordance between $[f]\#[g]$ and $[f']$. By construction, $F\# G$ can be smoothed, therefore $[f']=[f]\#[g]$. Cf. \cite[An alternative definition of the Kreck invariant]{Sk08}.
\end{proof}

\begin{proof}[Proof of Theorem \ref{thm:mainpl}]
Part (I) follows from Part (I) of Theorem \ref{thm:main} by (1) and (3).
Part (II) follows from Part (II) of Theorem \ref{thm:main} by (1), (2), and (3).
Part (III) follows from Part (II) of Theorem \ref{thm:main} by (1), (2), (3), and (4).
\end{proof}

\bigskip
\section{Proof of the main theorem modulo lemmas.}
\subsection{Plan of the proof.}
In this section we prove the main theorem modulo Surjectivity Lemma \ref{lem:surj}, Bijectivity Lemma \ref{lem:bij}, Preimage Lemma \ref{lem:pre}, Calculation Lemma \ref{lem:calc}, Linking Lemma \ref{lem:Linking}, and Claim \ref{clm:Wpre}. All of these statements are proved later in the corresponding sections.   

The plan of the proof is explained by the diagram in Fig.\ref{f:diagram}. In this subsection we only give informal explanations. All the new objects and statements mentioned here or in the diagram are rigorously defined or stated later in this section.

We represent $M_1$ as the result of cutting several solid tori from $S^3$ and then pasting them back together by the diffeomorphism exchanging parallels with meridians. By $\wh{M_1}$ we denote the compliment in $S^3$ to the solid tori, i.e., what is left of $S^3$ after cutting the tori and before pasting them back. The definition of $\wh{M_2}$ is analogous.

By $\wh{E}^6(\wh{M_1}\sqcup\wh{M_2})$ we denote the set of fixed on the boundary isotopy classes of proper embeddings $\wh{M_1}\sqcup\wh{M_2}\rightarrow D^6_+$. Given a representative of an element of $\wh{E}^6(\wh{M_1}\sqcup\wh{M_2})$ we can extend it in two different ``standard'' ways to either an embedding $S^3\sqcup S^3\rightarrow S^6$ or an embedding $M_1\sqcup M_2\rightarrow S^6$. These extensions define the maps $\sigma_R$ and $\sigma$ in the diagram.

It turns out that the map $\sigma$ (and $\sigma_R$) is surjective, see the Surjectivity Lemma \ref{lem:surj}. I.e., any embedding $M_1\sqcup M_2\rightarrow S^6$ is isotopic to a so-called ``standardized'' embedding which is ``standard'' on the solid tori and which maps $\wh{M_1}\sqcup\wh{M_2}$ to $D^6_+$. The proof of Surjectivity Lemma \ref{lem:surj} essentially repeats the proof of the first part of the Standardization Lemma in \cite{Sk15} (which is stated in slightly less general case than we require). 

Two isotopic ``standardized'' embeddings are not necessarily isotopic through ``standardized'' embeddings. This means that the map $\sigma$ is not bijective (and that the second part of the Standardization Lemma of \cite{Sk15} fails in the dimensions we are working in). By studying the geometric obstruction to the ``standardization'' of an isotopy between two ``standardized'' embeddings we prove the Preimage Lemma \ref{lem:pre}.

The set $E^6(S^3\sqcup S^3)$ is known and the maps $\sigma$ and $\sigma_R$ are surjective. Therefore we can classify the unknown set $E^6(M_1\sqcup M_2)$ by describing the (not well-defined) ``composition'' $\sigma_R\circ\sigma^{-1}$. This task is accomplished by the Bijectivity, Preimage, and Calculation Lemmas \ref{lem:bij}, \ref{lem:pre}, and \ref{lem:calc}.

	\begin{figure}[H]
	\centering
	\includegraphics{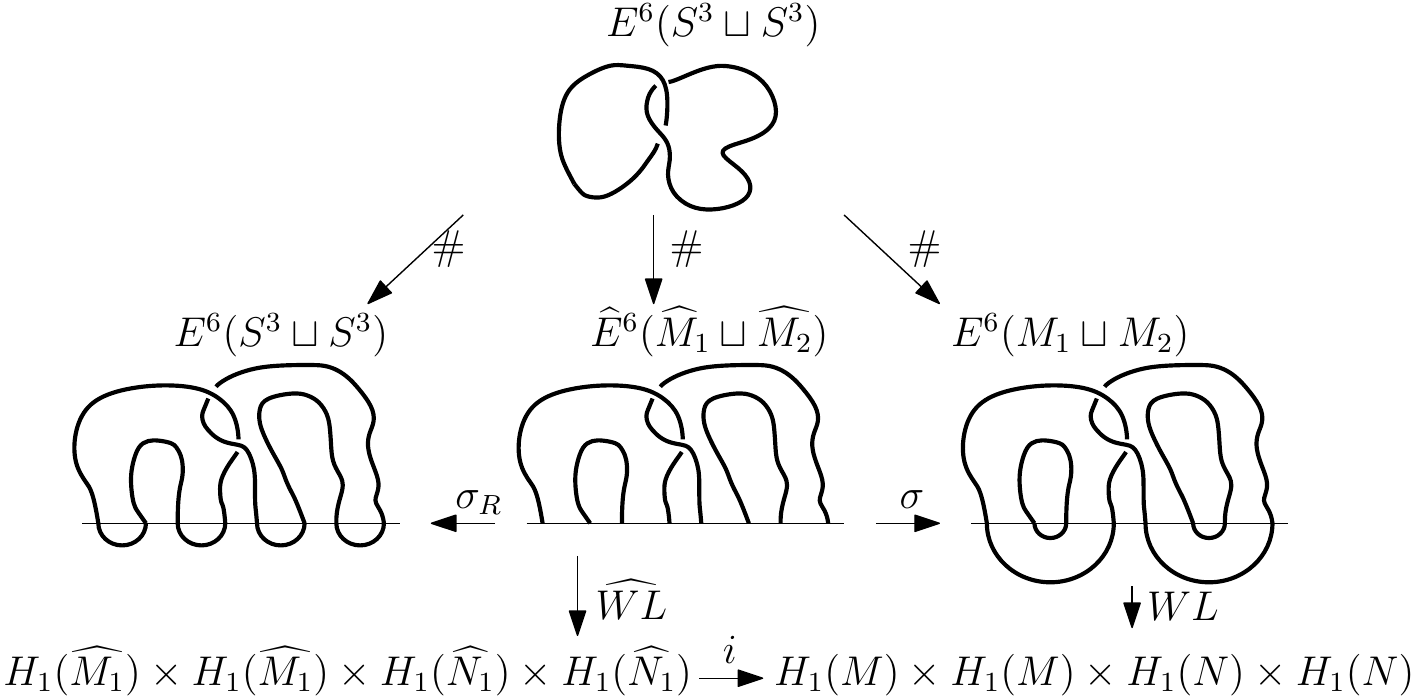}
	\caption{The diagram.}
	\label{f:diagram}
	\end{figure}

\smallskip
\subsection{Definitions of $T_n, P, \wh{M_k},m,R$.} 
In this subsection we represent manifolds $M_1$ and $M_2$ as results of a surgery of $S^3$ on several embedded circles.

For any $n>0$ let $$T_n:=\underbrace{S^1\times D^2\sqcup\ldots\sqcup S^1\times D^2}_n$$ be the disjoint union of $n$ copies of $S^1\times D^2$. 

Let $$R:S^1\times S^1\rightarrow S^1\times S^1$$ be the diffeomorphism exchanging the parallel with the meridian.

By \cite[end of \S 12, beginning of \S 14]{PS} for each $k\in\{1,2\}$ there are $m_k>0$ and an embedding $P_k:T_{m_k}\rightarrow S^3$ such that
\begin{itemize}
\item the restriction of $P_k$ to each of $m_k$ connected components of $T_{m_k}$ is isotopic to the standard embedding $S^1\times D^2\rightarrow S^3$;
\item if we denote $$\wh{M_k}:=\text{\quad the closure of\quad} S^3\setminus P_k(T_{m_k})$$
then 
$$M_k \cong \wh{M_k}\underset{P_k(x)=R(x), x\in \partial T_{m_k}}{\bigcup}T_{m_k},$$
(where ``$\cong$'' is a diffeomorphism).
\end{itemize}
For the rest of the text and for each $k\in\{1,2\}$ we replace $M_k$ with
$$\wh{M_k}\underset{P_k(x)=R(x), x\in \partial T_{m_k}}{\bigcup}T_{m_k},$$
see Fig.\ref{f:surgery}.

	\begin{figure}[H]
	\begin{center}
	\includegraphics{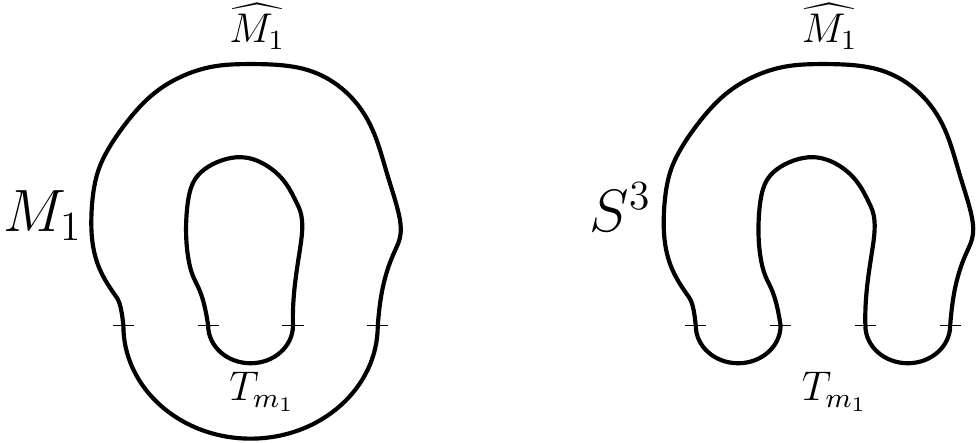}
	\caption{Manfiolds $M_1$ on the left and $S^3$ on the right.}
	\label{f:surgery}
	\end{center}
	\end{figure}
	
Until the end of the text $k\in\{1,2\}$ and $1\leq i\leq m_k$. I.e., all the statements involving $k$ and/or $i$ are given {\it for all} $k\in\{1,2\}$ {\it and} $1\leq i\leq m_k$, unless specifically said otherwise.

\smallskip
\subsection{Definitions of $P_{k,i}, \p_{k,i}, h$.} 
Denote by $P_{k,i}$ the restriction of $P_k$ to the $i$-th connected component.

Fix an orientation of $S^1\times D^2$. Consider the meridian $\m:=*\times S^1\subset S^1\times D^2$ with some orientation. Construct a normal framing of $\m$ in the following way. The first vector of the framing ``looks'' inside the full-torus $S^1\times D^2$, the second vector of the framing is then determined uniquely by the compatibility with orientation. Denote the obtained framed circle by the same letter $\m$.

Define framed circles $\p_{k,i}\subset S^3$ by the formula
$$\p_{k,i} := P_{k,i}R\m\subset S^3.$$


Let $a\subset S^3$ be any framed $1$-submanifold. Shift $a$ slightly along the first vector of its framing and denote the obtained submanifold by $a'$. The {\it Hopf invariant} $h(a)$ of $a$ is defined by the formula  
$$h(a):=\lk(a,a')\in{\mathbb Z}.$$

The following claim easily follows from the definition of $\p_{k,i}$.
\begin{clm}
\label{clm:untwist}
For any $k\in\{1,2\}$ and $1\leq i\leq m_k$ we have $h(\p_{k,i})=0$.
\end{clm}

\smallskip
\subsection{Definition of the set $\wh{E}^6(\wh{M_1}\sqcup\wh{M_2})$.}
Denote by $D^6_+$ and $D^6_-$ the northern and the southern hemispheres of $S^6$, respectively (the exact choice of the ``north'' and ``south'' poles is not important). 

Let $$s_k:\underbrace{D^2\times D^4\sqcup\ldots\sqcup D^2\times D^4}_{m_k}\rightarrow D^6_-$$ be an embedding such that
\begin{itemize}
\item its restriction to each $*\times D^4$ is isotopic to the standard proper embedding $D^4\rightarrow D^6_-$,
\item there are pairwise disjoint $6$-balls $B_{k,i}\subset D^6_-$ such that the $s_k$-image of the $i$-th connected component lie in $B_{k,i}$.
\end{itemize}
We additionally demand that for every $1\leq i\leq m_1$, $1\leq j\leq m_2$ the balls $B_{1,i}$ and $B_{2,j}$ are disjoint. 

Denote by $B^{\square}_{k,i}$ some tubular neighbourhood of $s_{k,i}(D^2\times D^4)$ in $B_{k,i}$ modulo $s_{k,i}(D^2\times S^3)$. Note that $B^{\square}_{k,i}$ is a manifold with ``corners'' diffeomorphic to $D^2\times D^4$.

We consider $S^1\times D^2$ as a submanidold of $D^2\times D^4$ where the inclusion $S^1\times D^2\subset D^2\times D^4$ is given by the obvious inclusions $S^1=\partial D^2\subset D^2$ and $D^2\subset D^4$.

Denote by $\wh{E}^6(\wh{M_1}\sqcup\wh{M_2})$ the set of isotopy classes fixed on the boundary, of proper embeddings $\wh{f}:\wh{M_1}\sqcup\wh{M_2}\rightarrow D^6_+$ such that $$\wh{f}\circ P_k|_{\partial T_{m_k}} = s_k|_{\partial T_{m_k}}$$ for each $k\in\{1,2\}$.

\smallskip
\subsection{Definition of operations $\sigma$, $\sigma_R$ and the action $\#$.}
For an embedding $\wh{f}:\wh{M_1}\sqcup\wh{M_2}\rightarrow D^6_+$ such that $[\wh{f}]\in \wh{E}^6(\wh{M_1}\sqcup\wh{M_2})$ define
\[ 
\sigma (\wh{f}):M_1\sqcup M_2\rightarrow S^6 \text{\quad by \quad} \sigma(\wh{f})(x):=\begin{cases}
\wh{f}(x) \text{ if } x\in\wh{M_1}\sqcup\wh{M_2}\\
s_k(x) \text{ if } x\in (M_k\setminus \wh{M_k})=T_{m_k}
\end{cases},
\]
see Fig.\ref{f:sigma}.

	\begin{figure}[H] 
	\begin{center} 
	\includegraphics{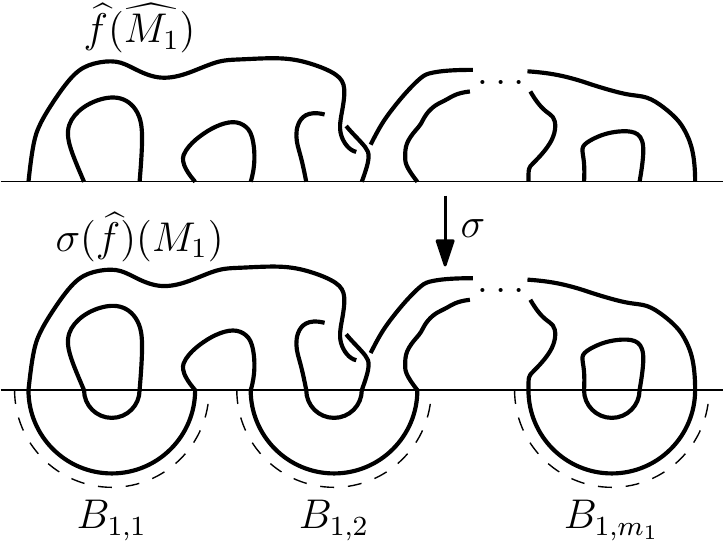} 
	\end{center} 
	\caption{Operation $\sigma$ (only $M_1$ is shown).}
	\label{f:sigma} 
	\end{figure}

Denote by $s_{k,i}$ the restrictions of $s_k$ to the $i$-th connected component.

Let $A_R:D^6_-\rightarrow D^6_-$ be an orientation preserving diffeomorphism such that
\begin{itemize}
\item $A_R(B_{k,i})=B_{k,i}$,
\item $A_R\circ s_{k,i}|_{S^1\times S^1}= s_{k,i}|_{S^1\times S^1}\circ R$.
\end{itemize}
Such a diffeomorphism exists because all embeddings $S^1\times S^1\rightarrow\partial D^6_-$ are isotopic and because smooth isotopies are ambient, see \cite[Theorem $2.1$]{Hu70}.

Denote
$$s_{R,k}:=A_R\circ s_k \text{\quad and\quad} s_{R,k,i}:=A_R\circ s_{k,i}.$$

For an embedding $\wh{f}:\wh{M_1}\sqcup\wh{M_2}\rightarrow D^6_+$ such that $[\wh{f}]\in \wh{E}^6(\wh{M_1}\sqcup\wh{M_2})$ define
\[ 
\sigma_R (\wh{f}):M_1\sqcup M_2\rightarrow S^6 \text{\quad by \quad} \sigma(\wh{f})(x):=\begin{cases}
\wh{f}(x) \text{ if } x\in\wh{M_1}\sqcup\wh{M_2}\\
s_{R,k}(x) \text{ if } x\in (M_k\setminus \wh{M_k})=T_{m_k}
\end{cases},
\]
see Fig.\ref{f:sigma_r}.

	\begin{figure}[H] 
	\begin{center} 
	\includegraphics{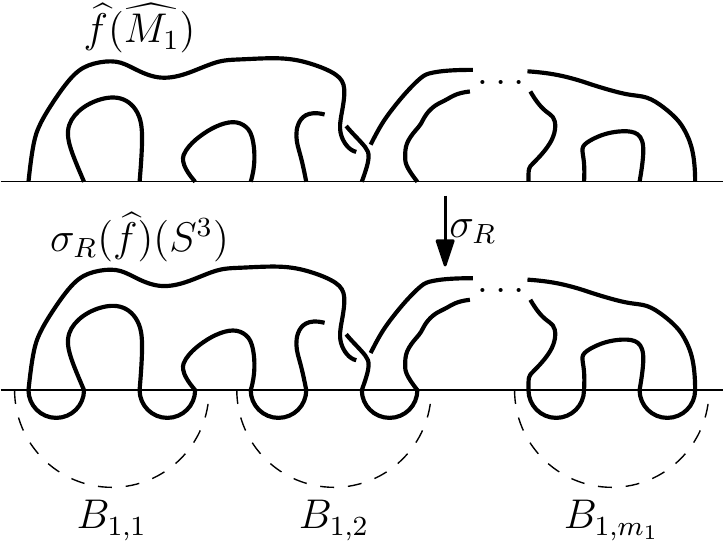} 
	\end{center} 
	\caption{Operation $\sigma_R$ (only $M_1$ is shown).}
	\label{f:sigma_r} 
	\end{figure}
	
Clearly, if $[\wh{f}]=[\wh{f'}]$ for some other embedding $\wh{f'}$, then $[\sigma(\wh{f})]=[\sigma(\wh{f'})]$ and $[\sigma_R(\wh{f})]=[\sigma_R(\wh{f'})]$. Therefore $\sigma$ and $\sigma_R$ 
induce well-defined maps 
$$E^6(S^3\sqcup S^3) \xleftarrow{\sigma_R} \wh{E}^6(\wh{M_1}\sqcup\wh{M_2}) \xrightarrow{\sigma} E^6(M_1\sqcup M_2),$$
which we denote by the same letters. 

Note that in our notation $\sigma(\wh{f})$ is an {\it embedding} while $\sigma([\wh{f}])$ is an {\it isotopy class}.

The group $E^6(S^3\sqcup S^3)$ acts on each of the sets $E^6(S^3\sqcup S^3)$, $\wh{E}^6(\wh{M_1}\sqcup\wh{M_2})$, and $E^6(M_1\sqcup M_2)$ via the component-wise connected sum $\#$. The action on $\wh{E}^6(\wh{M_1}\sqcup\wh{M_2})$ is defined analogously to the action on $E^6(S^3\sqcup S^3)$ or $E^6(M_1\sqcup M_2)$.

The following claim easily follows from the definitions of $\sigma$, $\sigma_R$, and the embedded connected sum action $\#$.
\begin{clm}[$\#$-commutativity]
\label{clm:comm}
The embedded connected sum action $\#$ commutes with $\sigma$ and $\sigma_R$. I.e., for any isotopy classes $[\wh{f}]\in\wh{E}^6(\wh{M_1}\sqcup\wh{M_2})$ and $[g]\in E^6(S^3\sqcup S^3)$ we have $\sigma([\wh{f}]\#[g])=\sigma([\wh{f}])\#[g]$ and $\sigma_R([\wh{f}]\#[g])=\sigma_R([\wh{f}])\#[g]$.
\end{clm}

\begin{lem}[Surjectivity]
\label{lem:surj}
Maps $\sigma$ and $\sigma_R$ are surjective.
\end{lem}

\smallskip
\subsection{Definition of the Whitney invariants $\wh{W}_k$, $\wh{\K}_k$ of proper embeddings.}
The definition of 
$$\wh{W}_k,\wh{\K}_k:\wh{E}^6(\wh{M_1}\sqcup\wh{M_2})\rightarrow H_1(\wh{M_k})$$
is analogous to the definition of
$$W_k,\K_k:E^6(M_1\sqcup M_2)\rightarrow H_1(M_k).$$
One needs only to replace ``homotopy'' by ``homotopy relative to the boundary'' and define a ``base embedding'' $\wh{f}^0:\wh{M_1}\sqcup\wh{M_2}\rightarrow D^6_+$. To do the latter we choose some $[\wh{f}^0]\in \wh{E}^6(\wh{M_1}\sqcup\wh{M_2})$ such that $\sigma([\wh{f}^0])=[{f^0}]$. The existence of such $[\wh{f}^0]$ is guaranteed by Surjectivity Lemma \ref{lem:surj}.

The following claim easily follows from the definition of $\wh{\K}_k$.
\begin{clm}
\label{clm:WZeifert}
Take any $k\in\{1,2\}$ and $[\wh{f}]\in\wh{E}^6(\wh{M_1}\sqcup\wh{M_2})$. Let $\Delta_k\subset D^6_+$ be a proper submanifold ``with corners'', $\partial\Delta_k= \wh{f}(\wh{M_k})\cup(\partial\Delta_k\cap\partial D^6_+)$. Suppose that $\Delta_k$ is disjoint with $\wh{f}(\partial\wh{M_{3-k}})\subset\partial D^6_+$. Then
$$\wh{L}_{3-k}(\wh{f})=[(\wh{f}^{-1})\Delta_k]\in H_1(\wh{M_{3-k}}).$$
\end{clm}

For brevity, denote 
$$\wh{\WL}:=\wh{W}_1\times \wh{\K}_1\times \wh{W}_2\times \wh{\K}_2.$$

The map 
$$H_1(\wh{M_1})\times H_1(\wh{M_1})\times H_1(\wh{M_2})\times H_1(\wh{M_2})\rightarrow H_1({M_1})\times H_1({M_1})\times H_1({M_2})\times H_1({M_2})$$
in the diagram is induced by the inclusions $\wh{M_1}\subset M_1$ and $\wh{M_2}\subset M_2$.

Our choice of the ``base element'' $[\wh{f}^0]\in \wh{E}^6(\wh{M_1}\sqcup\wh{M_2})$ implies the following two claims.
\begin{clm}
\label{clm:Wrep}
For any $k\in\{1,2\}$, $[\wh{f}]\in\wh{E}^6(\wh{M_1}\sqcup\wh{M_2})$ and $[f]:=\sigma([\wh{f}])$ the homology classes $\wh{W}_k(\wh{f})
$ and $W_k(f)$ can be represented by the same $1$-submanifold in $\wh{M_k}$. Likewise, the homology classes $\wh{\K}_k(\wh{f})
$ and $\K_k(f)$ can be represented by the same $1$-submanifold in $\wh{M_k}$.
\end{clm}
\begin{clm}
\label{clm:commW}
The square in the diagram above commutes.
\end{clm}

\smallskip
\subsection{Proof of part (II) of Theorem \ref{thm:main}.}
\begin{lem}[Bijectivity]
\label{lem:bij}
For any $x\in H_1(\wh{M_1})\times H_1(\wh{M_1})\times H_1(\wh{M_2})\times H_1(\wh{M_2})$ the restriction $\sigma_R|_{\wh{\WL}^{-1}(x)}$ is a bijection.
\end{lem}

\begin{clm}
\label{clm:Wpre}
Let $[f],[f']\in E^6(M_1\sqcup M_2)$ be isotopy classes such that $\WL(f)=\WL(f')$. Then there are isotopy classes $[\wh{f}],[\wh{f'}]\in\wh{E}^6(\wh{M_1}\sqcup \wh{M_2})$ such that 
$\sigma([\wh{f}])=[f]$, $\sigma([\wh{f'}])=[f']$, and $\wh{\WL}(\wh{f})=\wh{\WL}(\wh{f'})$.
\end{clm}

\begin{proof}[Proof of part (II) of Theorem \ref{thm:main}]
Let $[f],[f']\in E^6(M_1\sqcup M_2)$ be isotopy classes such that $\WL(f)=\WL(f')$. To complete the proof we need to find an embedding $g:S^3\sqcup S^3\rightarrow S^6$ such that $[f']\#[g]=[f]$.

Let $[\wh{f}],[\wh{f'}]\in\wh{E}^6(\wh{M_1}\sqcup \wh{M_2})$ be the isotopy classes whose existence is guaranteed by Claim \ref{clm:Wpre}. By Haefliger Theorem \ref{thm:H} there is an embedding $g:S^3\sqcup S^3\rightarrow D^6_+$ such that
$$\sigma_R([\wh{f'}])\#[g]=\sigma_R([\wh{f}]).$$
By $\#$-commutativity Claim \ref{clm:comm} we get
$$\sigma_R([\wh{f'}]\#[g])=\sigma_R([\wh{f}]).$$
Clearly, $\wh{\WL}([\wh{f'}]\#[g])=\wh{\WL}([\wh{f'}])$, so by Bijectivity Lemma \ref{lem:bij} we have
$$[\wh{f'}]\#[g]=[\wh{f}].$$
So,
\begin{multline*}
[\wh{f'}]\#[g]=[\wh{f}]\text{\quad}\Rightarrow\text{\quad} \sigma([\wh{f'}]\#[g])=\sigma([\wh{f}])\overset{(1)}{\text{\quad}\Rightarrow\text{\quad}}\\
\overset{(1)}{\text{\quad}\Rightarrow\text{\quad}} \sigma([\wh{f'}])\#[g] = \sigma([\wh{f}]) \overset{(2)}{\text{\quad}\Rightarrow\text{\quad}} [f']\#[g]=[f],
\end{multline*}
where (1) follows by the $\#$-commutativity Claim \ref{clm:comm} and (2) follows from the choice of $[\wh{f}]$ and $[\wh{f'}]$. We get that $g$ is as required.
\end{proof}

\smallskip
\subsection{Definition of $\omega$.}
Define
$$\omega_{k,i}:S^3\rightarrow S^6 \text{\quad by the formula \quad} \omega_{k,i}:=s_{k,i}|_{0\times S^3},$$
see Fig.\ref{f:omega}.

	\begin{figure}[H]
	\begin{center} 
	\includegraphics{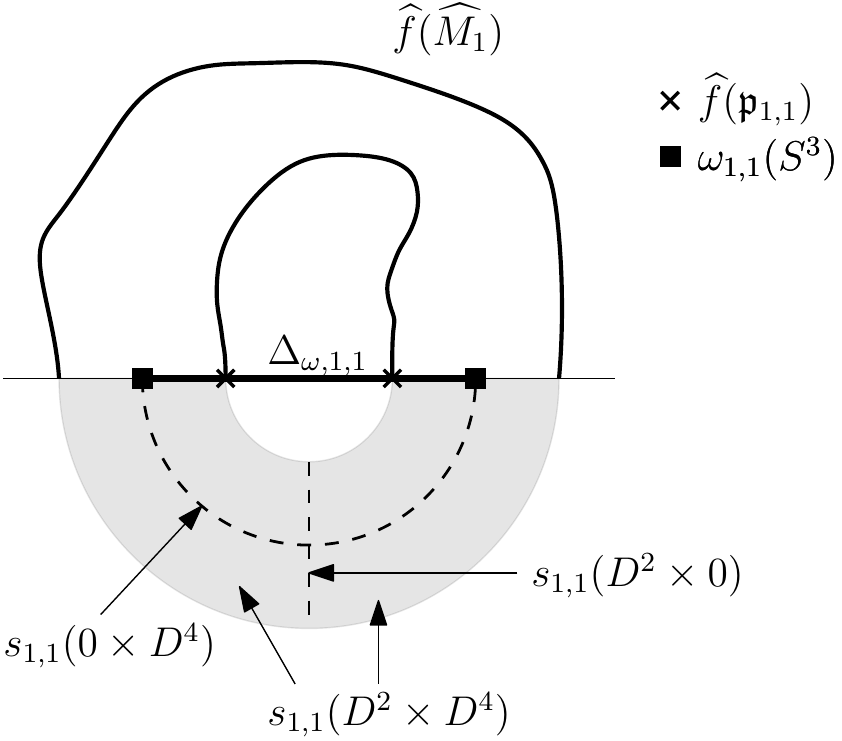} 
	\end{center} 
	\caption{The circle $f(\p_{1,1})$, the sphere $\omega_{1,1}(S^3)$, and the disk $\Delta_{\omega,1,1}$.}
	\label{f:omega} 
	\end{figure}

\smallskip
\subsection{Multiple linked embedded connected sum.}
Take any $[\wh{f}]\in \wh{E}^6(\wh{M_1}\sqcup\wh{M_2})$ and $g:S^3\rightarrow \partial D^6_+$ such that the images of $\wh{f}$ and $g$ are disjoint. For $k\in\{1,2\}$ we shall write $$\wh{f}\#_kg$$ meaning $\wh{f}\#_kg'$ -- the linked embedded connected sum of $f$ with some embedding $g':S^3\rightarrow {\rm Int}D^6_+$ obtained from $g$ by a slight shift into the interior of $D^6_+$. This agreement guarantees that $[\wh{f}\#_kg]\in \wh{E}^6(\wh{M_1}\sqcup\wh{M_2})$.

For any integer $a$ we denote
\begin{itemize}
\item $\wh{f}\#_kag:=\wh{f}\underbrace{\#_kg\#_kg\ldots\#_kg}_{a}, \text{\quad if \quad} a>0,$
\item $\wh{f}\#_kag:=\wh{f}\#_k(-a)(-g), \text{\quad if \quad} a<0,$
\item $\wh{f}\#_kag:=\wh{f}\text{\quad if \quad} a=0.$
\end{itemize}
Here $-g:S^3\rightarrow \partial D^6_+$ is an embedding such that ${\rm Im}(-g)={\rm Im}(g)$ and $[g]\#[-g]$ is trivial considered as an isotopy class of an embedding $S^3\rightarrow S^6$.

\smallskip
\subsection{Proof of part (III) of Theorem \ref{thm:main}.}
The following lemma allows us to describe the preimage of $\sigma$.

\begin{lem}[Preimage]
\label{lem:pre}
For any $[\wh{f}],[\wh{f}']\in \wh{E}^6(\wh{M_1}\sqcup\wh{M_2})$ we have that $\sigma([\wh{f}])=\sigma([\wh{f}'])$ if and only if
$$[\wh{f}']=[\wh{f}\underset{i=1}{\overset{m_1}{\#_1}}a_i\omega_{1,i} \underset{i=1}{\overset{m_1}{\#_2}}b_i\omega_{i,1} \underset{j=1}{\overset{m_2}{\#_2}}c_j\omega_{2,j} \underset{j=1}{\overset{m_2}{\#_1}}d_j\omega_{2,j}]$$ 
for some integers $a_i$, $b_i$, $c_j$, and $d_j$.
\end{lem}

\begin{rem}
In other words, the lemma states that $\sigma([\wh{f}])=\sigma([\wh{f}'])$ if and only if $[\wh{f}']$ can be obtained from $[\wh{f}]$ by several operations of the form
\begin{itemize}
\item $[\wh{f}]\to [\wh{f}\#_1\pm\omega_{1,i}],$
\item $[\wh{f}]\to [\wh{f}\#_2\pm\omega_{1,i}],$
\item $[\wh{f}]\to [\wh{f}\#_1\pm\omega_{2,j}],$
\item $[\wh{f}]\to [\wh{f}\#_2\pm\omega_{2,j}].$
\end{itemize}
where $1\leq i\leq m_1$ and $1\leq j\leq m_2$.
\end{rem}

\begin{proof}[Proof of the ``if part'' of Preimage Lemma \ref{lem:pre}]
The remark above makes the ``if'' part obvious. For instance, there is an isotopy between $\sigma(\wh{f}\#_1\pm\omega_{1,i})$ and $\sigma(\wh{f})$ which ``drags'' the sphere $\omega_{1,i}(S^3)$ along the disk $s_{1,i}(0\times D^4)$. This is indeed an isotopy because the disk $s_{1,i}(0\times D^4)$ is disjoint with ${\rm Im}(\sigma(\wh{f}))$, see Fig.\ref{f:omega} for the case $i=1$.
\end{proof}

For a homology class $a\in H_1(\wh{M_k})$ we denote by $\lk(\p_{k,i},a)$ the linking number of $\p_{k,i}\subset \partial\wh{M_k}$ and any oriented $1$-submanifold of ${\rm Int}\wh{M_k}\subset S^3$ representing $a$. Clearly, this linking number is well defined, i.e., do not depend on the choice of the representative.

Denote by $[\p_{k,i}]$ the respective homology class in $H_1(\wh{M_k})$.

Let $\wh{f}$ be a proper embedding such that $[\wh{f}]\in \wh{E}^6(\wh{M_1}\sqcup\wh{M_2})$. Denote 
$$\x_{k,i}(\wh{f}):=\lx(\omega_{k,i},(\sigma_R\wh{f})_k),$$
where $(\sigma_R\wh{f})_k$ is the restriction of $\sigma_R\wh{f}:S^3\sqcup S^3\rightarrow S^6$ to the $k$-th connected component of its domain.

\begin{lem}[Calculation]
\label{lem:calc}
Suppose that $[\wh{f}]\in \wh{E}^6(\wh{M_1}\sqcup\wh{M_2})$, $1\leq i \leq m_1$, and $1\leq j \leq m_2$.

In the first column of the table is an embedding $\wh{f}'$. In the first row are symbols denoting different isotopy invariants.

In each cell of the columns ``$\li$'' to ``$r_2$'' is the difference of the corresponding invariant of $\sigma_R(\wh{f}')$ and $\sigma_R(\wh{f})$.

In each cell of the columns ``$\wh{W}_1$'' to ``$\wh{\K}_2$'' is the difference of the corresponding invariant of $\wh{f}'$ and $\wh{f}$.

\begin{center}
{
\noindent\makebox[\linewidth]
{
\begin{tabular}{| l || l | l | l | l | l | l | l | l |}
\hline
$\wh{f}'$ & $\li$ & $\lii$ & $r_1$ & $r_2$ & $\wh{W}_1$ & $\wh{W}_2$ & $\wh{\K}_1$ & $\wh{\K}_2$ \\ \hline \hline
$\wh{f}\#_1\omega_{1,i}$ & $0$ & $2\lk(\wh{\K}_1\wh{f},\p_{1,i})$ & $\frac{\x_{1,i}(\wh{f})}{2}$ & $0$ & $[\p_{1,i}]$ & $0$ & $0$ & $0$ \\ \hline
$\wh{f}\#_2\omega_{1,i}$ & $2\lk(\wh{\K}_1\wh{f},\p_{1,i})$ & $\x_{1,i}(\wh{f})$ & $0$ & $0$ & $0$ & $0$ & $[\p_{1,i}]$ & $0$ \\ \hline
$\wh{f}\#_2\omega_{2,j}$ & $2\lk(\wh{\K}_2\wh{f},\p_{2,j})$ & $0$ & $0$ & $\frac{\x_{2,j}(\wh{f})}{2}$ & $0$ & $[\p_{2,j}]$ & $0$ & $0$ \\ \hline
$\wh{f}\#_1\omega_{2,j}$ & $\x_{2,j}(\wh{f})$ & $2\lk(\wh{\K}_2\wh{f},\p_{2,j})$ & $0$ & $0$ & $0$ & $0$ & $0$ & $[\p_{2,j}]$ \\ \hline
\end{tabular}
}
}
\end{center}
\end{lem}

We shall refer to the cells of the table by their respective row number and column title. E.g., cell (1,$\li$) contains $0$ and means that $\li(\sigma_R(\wh{f}\#_1\omega_{1,i}))-\li(\sigma_R(\wh{f}))=0$; cell (3,$\wh{W}_2$) contains $[\p_{2,j}]$ and means that $\wh{W}_2(\wh{f}\#_2\omega_{2,j})-\wh{W}_2(\wh{f})=[\p_{2,j}]$; etc.

\begin{lem}[Linking]
\label{lem:Linking}
For any $k\in\{1,2\}$, integers $a_i$, and isotopy class $[\wh{f}]\in \wh{E}^6(\wh{M_1}\sqcup\wh{M_2})$ the following implication holds $$\overset{m_k}{\underset{i=1}{\Sigma}}a_i[\p_{k,i}]=0 \text{\quad}\Rightarrow\text{\quad} \overset{m_k}{\underset{i=1}{\Sigma}}a_i\x_{k,i}(\wh{f}) = \overset{m_k}{\underset{i=1}{\Sigma}}2\lk(\p_{k,i}, \wh{W}_k\wh{f}).$$
\end{lem}

Denote by $[\p_{k,i}]_{\partial\wh{M_k}}$ the respective homology class in $H_1(\partial\wh{M_k})$. 

Consider the following part of the Mayer--Vietoris long exact sequence $$H_2(M_k)\xrightarrow{\partial} H_1(\partial \wh{M_k})\xrightarrow{i_{\wh{M_k}}, i_{T_{m_k}}} H_1(\wh{M_k})\oplus H_k(T_{m_1}),$$
where the maps $i_{\wh{M_k}}$ and $i_{T_{m_k}}$ are induced by the inclusions $\partial\wh{M_k}\subset\wh{M_k}$ and $\partial\wh{M_k}\subset T_{m_k}$.

\begin{clm}
\label{clm:delta}
The image of $\partial:H_2(M_k)\rightarrow H_1(\partial \wh{M_k})$ is the subgroup of $H_1(\partial \wh{M_k})$ consisting of all linear combinations of the form $\underset{i=1}{\overset{m_k}{\Sigma}}a_i[\p_{k,i}]_{\partial\wh{M_k}}$ such that $\underset{i=1}{\overset{m_k}{\Sigma}}a_i[\p_{k,i}]=0\in H_1(\wh{M_k})$.
\end{clm}

\begin{proof}
From the construction of $M_k$ it is clear, that ${\rm Ker}(i_{T_{m_k}})$ consists exclusively of all linear combinations of $[\p_{k,i}]_{\partial\wh{M_k}}$.

By the definition, $i_{\wh{M_k}}([\p_{k,i}]_{\partial\wh{M_k}})=[\p_{k,i}]$, so any linear combination of the form $\underset{i=1}{\overset{m_k}{\Sigma}}a_i[\p_{k,i}]_{\partial\wh{M_k}}$ is in ${\rm Ker}(i_{\wh{M_k}})$ if and only if $\underset{i=1}{\overset{m_k}{\Sigma}}a_i[\p_{k,i}]=0\in H_1(\wh{M_k})$.

Now the claim follows from the exactness of the Mayer--Vietoris sequence above.
\end{proof}

\begin{clm}
\label{clm:alpha}
Take any $\alpha\in H_2(M_k)$. By Claim \ref{clm:delta}, $\partial\alpha=\underset{i=1}{\overset{m_k}{\Sigma}}a_i[\p_{k,i}]_{\partial\wh{M_k}}$ for some integers $a_i$. Then for any $[\wh{f}]\in \wh{E}^6(\wh{M_1}\sqcup\wh{M_2})$ and $[f]:=\sigma([\wh{f}])$
\begin{itemize}
\item [(I)] $\K_kf\cap\alpha=\underset{i=1}{\overset{m_k}{\Sigma}}a_i\lk(\wh{\K}_k\wh{f},\p_{k,i})$,
\item [(II)] $W_kf\cap\alpha=\underset{i=1}{\overset{m_k}{\Sigma}}a_i\frac{\x_{k,i}}{2}$.
\end{itemize}
\end{clm}

\begin{proof}
{\bf (I).} Follows from $$\K_kf\cap\alpha=\wh{\K}_k\wh{f}\cap\alpha=\lk(\wh{\K}_k\wh{f},\partial\alpha)=\underset{i=1}{\overset{m_k}{\Sigma}}a_i\lk(\wh{\K}_k\wh{f},\p_{k,i}).$$
The first equality holds by Claim \ref{clm:Wrep}. The second equality holds by the definition of $\lk$.

{\bf (II).} Follows from $$W_kf\cap\alpha=\wh{W}_k\wh{f}\cap\alpha=\lk(\wh{W}_k\wh{f},\partial\alpha)=\underset{i=1}{\overset{m_k}{\Sigma}}a_i\lk(\wh{W}_k\wh{f},\p_{k,i})=\underset{i=1}{\overset{m_k}{\Sigma}}a_i\frac{\x_{k,i}}{2}.$$
The first equality holds by Claim \ref{clm:Wrep}. The second equality holds by the definition of $\lk$. The last equality holds by Linking Lemma \ref{lem:Linking}, which we can apply because $\overset{m_k}{\underset{i=1}{\Sigma}}a_i[\p_{k,i}]=0$ by Claim \ref{clm:delta}.
\end{proof}

\begin{proof}[Proof of the ``if'' statement in part (III) of Theorem \ref{thm:main}]
Until the end of the proof identify $E^6(S^3\sqcup S^3)$ with $\wt{\Z^4}$ by the isomorphism $\li\times\lii\times r_1\times r_2$.

Let $f:M_1\sqcup M_2\rightarrow S^6$ be an embedding. Let $g:S^3\sqcup S^3\rightarrow S^6$ be an embedding such that $[g]\in {\rm Stab}_f\subset \wt{\Z^4}$. We need to prove that $[f]=[f]\#[g]$.

By the definition of ${\rm Stab}_f$, there are $\alpha,\beta\in H_2(M_1)$ and $\gamma,\delta\in H_2(M_2)$ such that $[g]=[g_\alpha]+[g_\beta]+[g_\gamma]+[g_\delta]$, where
\begin{itemize}
\item $[g_\alpha]=(0,2\K_1f\cap \alpha,W_1f\cap \alpha, 0)\in \wt{\Z^4}$,
\item $[g_\beta]=(2\K_1f\cap \beta,2W_1f\cap \beta,0, 0)\in \wt{\Z^4}$,
\item $[g_\gamma]=(2\K_2f\cap \gamma,0,0,W_2f\cap \gamma)\in \wt{\Z^4}$,
\item $[g_\delta]=(2W_2f\cap \delta,2\K_2f\cap \delta,0, 0)\in \wt{\Z^4}$.
\end{itemize}

It is enough to prove that $[f]=[f]\#[g_\alpha]$, $[f]=[f]\#[g_\beta]$, $[f]=[f]\#[g_\gamma]$, and $[f]=[f]\#[g_\delta]$. We shall only prove the first equality because the proofs of others are analogous.

By Claim \ref{clm:delta}, there are integers $a_i$ such that $\partial\alpha=\underset{i=1}{\overset{m_1}{\Sigma}}a_i[\p_{1,i}]_{\partial\wh{M_1}}$. By Surjection Lemma \ref{lem:surj}, there is an embedding $\wh{f}:\wh{M_1}\sqcup\wh{M_2}\rightarrow D^6_+$ such that $\sigma([\wh{f}])=[f]$. Denote $$[\wh{f'}]:=[\wh{f}\underset{i=1}{\overset{m_1}{\#_1}}a_i\omega_{1,i}].$$ 
Now the equality $[f]=[f]\#[g_\alpha]$, which we want to prove, follows from 
$$[f]=\sigma([\wh{f}])\overset{(1)}{=}\sigma([\wh{f'}])\overset{(2)}{=}\sigma([\wh{f}]\# [g_\alpha])\overset{(3)}{=}\sigma([\wh{f}])\# [g_\alpha]=[f]\#[g_\alpha],$$
where (1) follows by Preimage Lemma \ref{lem:pre} and (3) follows by $\#$-commutativity Claim \ref{clm:comm}. Equation (2) follows from
$$\wh{\WL}([\wh{f'}])\overset{(4)}{=}\wh{\WL}([\wh{f}]\# [g_\alpha])$$
and
$$\sigma_R([\wh{f'}])\overset{(5)}{=}\sigma_R([\wh{f}]\# [g_\alpha])$$
by Bijection Lemma \ref{lem:bij}. It remains to prove (4) and (5).

Now (4) follows from
$$\wh{\WL}([\wh{f'}])-\wh{\WL}([\wh{f}]\# [g_\alpha])=\wh{\WL}([\wh{f'}])-\wh{\WL}([\wh{f}])=\underset{i=1}{\overset{m_1}{\Sigma}}a_i[\p_{1,i}]=0,$$
where the second equality follows by the definiton of $[\wh{f'}]$ and Calculation Lemma \ref{lem:calc}, cells (1,$\wh{W}_1$-$\wh{\K}_2$). The last equality holds by Claim \ref{clm:delta}.

And (5) follows from
\begin{multline*}
\sigma_R([\wh{f'}])-\sigma_R([\wh{f}]\# [g_\alpha])=\sigma_R([\wh{f'}])-\sigma_R([\wh{f}])-[g_\alpha]=\\
=(0, 2\underset{i=1}{\overset{m_1}{\Sigma}}a_i\lk(\wh{\K}_1\wh{f},\p_{1,i}), \underset{i=1}{\overset{m_1}{\Sigma}}a_i\frac{\x_{1,i}(\wh{f})}{2}, 0)-[g_\alpha]=\\
=(0,2\K_1f\cap \alpha,W_1f\cap \alpha, 0)-(0,2\K_1f\cap \alpha,W_1f\cap \alpha, 0)=0\in\wt{\Z^4},
\end{multline*}
where the second equality holds by Calculation Lemma \ref{lem:calc}, cells (1,$\li$-$r_2$). The third equality holds by Claim \ref{clm:alpha} and by the definition of $[g_\alpha]$.
\end{proof}

\begin{proof}[Proof of the ``only if'' statement in part (III) of Theorem \ref{thm:main}]
Until the end of the proof identify $E^6(S^3\sqcup S^3)$ with $\wt{\Z^4}$ by the isomorphism $\li\times\lii\times r_1\times r_2$.

Let $f:M_1\sqcup M_2\rightarrow S^6$ be an embedding. Let $g:S^3\sqcup S^3\rightarrow S^6$ be an embedding such that $[f]=[f]\#[g]$. We need to prove that $[g]\in {\rm Stab}_f$.

By Surjection Lemma \ref{lem:surj}, there is an embedding $\wh{f}:\wh{M_1}\sqcup\wh{M_2}\rightarrow D^6_+$ such that $\sigma([\wh{f}])=[f]$. By $\#$-commutativity Claim \ref{clm:comm}, $\sigma([\wh{f}]\#[g])=\sigma([\wh{f}])\#[g]=[f]\#[g]=[f]$. 

So both $[\wh{f}]$ and $[\wh{f}]\#[g]$ are $\sigma$-preimages of $[f]$. By Preimage Lemma \ref{lem:pre}, there are integers $a_i$, $b_i$, $c_j$, and $d_j$ such that $$[\wh{f}]\#[g]=[\wh{f}\underset{i=1}{\overset{m_1}{\#_1}}a_i\omega_{1,i} \underset{i=1}{\overset{m_1}{\#_2}}b_i\omega_{i,1} \underset{j=1}{\overset{m_2}{\#_2}}c_j\omega_{2,j} \underset{j=1}{\overset{m_2}{\#_1}}d_j\omega_{2,j}].$$ 
Clearly $\wh{\WL}([\wh{f}])=\wh{\WL}([\wh{f}]\#[g])$. From that, the equation above, and Calculation Lemma \ref{lem:calc} (last four columns of the table) we get that $$\underset{i=1}{\overset{m_1}{\Sigma}}a_i[\p_{1,i}]=\underset{i=1}{\overset{m_1}{\Sigma}}b_i[\p_{1,i}]=\underset{j=1}{\overset{m_2}{\Sigma}}c_j[\p_{2,j}]=\underset{j=1}{\overset{m_2}{\Sigma}}d_j[\p_{2,j}]=0.$$

By Claim \ref{clm:delta}, there is $\alpha\in H_2(M_1)$ such that $\partial\alpha=\underset{i=1}{\overset{m_1}{\Sigma}}a_i[\p_{1,i}]_{\partial\wh{M_1}}$. The definitions of $\beta\in H_2(M_1)$, $\gamma,\delta\in H_2(M_2)$ are analogous but with $(a_i, \p_{1,i})$ replaced by $(b_i, \p_{1,i})$, $(c_j, \p_{2,j})$, and $(d_j, \p_{2,j})$, respectively.

The statement of the theorem now follows from
\begin{multline*}
\sigma_R([\wh{f}])+[g] \overset{(1)}{=} \sigma_R([\wh{f}]\#[g])=\\
=\sigma_R([\wh{f}\underset{i=1}{\overset{m_1}{\#_1}}a_i\omega_{1,i} \underset{i=1}{\overset{m_1}{\#_2}}b_i\omega_{i,1} \underset{j=1}{\overset{m_2}{\#_2}}c_j\omega_{2,j} \underset{j=1}{\overset{m_2}{\#_1}}d_j\omega_{2,j}]) \overset{(2)}{=}\\
\overset{(2)}{=}\sigma_R([\wh{f}\underset{i=1}{\overset{m_1}{\#_1}}a_i\omega_{1,i} \underset{i=1}{\overset{m_1}{\#_2}}b_i\omega_{i,1} \underset{j=1}{\overset{m_2}{\#_2}}c_j\omega_{2,j}])+(2W_2f\cap \delta,2\K_2f\cap \delta,0, 0) \overset{(3)}{=}\\
\overset{(3)}{=}\sigma_R([\wh{f}\underset{i=1}{\overset{m_1}{\#_1}}a_i\omega_{1,i} \underset{i=1}{\overset{m_1}{\#_2}}b_i\omega_{i,1}]) + (2\K_2f\cap \gamma,0,0,W_2f\cap \gamma) + (2W_2f\cap \delta,2\K_2f\cap \delta,0, 0) \overset{(4)}{=}\\
\overset{(4)}{=}\sigma_R([\wh{f}\underset{i=1}{\overset{m_1}{\#_1}}a_i\omega_{1,i}]) + (2\K_1f\cap \beta,2W_1f\cap \beta,0, 0) + \\ 
(2\K_2f\cap \gamma,0,0,W_2f\cap \gamma) + (2W_2f\cap \delta,2\K_2f\cap \delta,0, 0) \overset{(5)}{=}\\
\overset{(5)}{=}\sigma_R([\wh{f}]) + (0,2\K_1f\cap \alpha,W_1f\cap \alpha, 0)+(2\K_1f\cap \beta,2W_1f\cap \beta,0, 0)+\\
+(2\K_2f\cap \gamma,0,0,W_2f\cap \gamma)+(2W_2f\cap \delta,2\K_2f\cap \delta,0, 0) \\
\Downarrow \\
[g]=(0,2\K_1f\cap \alpha,W_1f\cap \alpha, 0)+(2\K_1f\cap \beta,2W_1f\cap \beta,0, 0)+\\
+(2\K_2f\cap \gamma,0,0,W_2f\cap \gamma)+(2W_2f\cap \delta,2\K_2f\cap \delta,0, 0)\in {\rm Stab}_f.
\end{multline*}
Here (1) follows by $\#$-commutativity Claim \ref{clm:comm}. It remains to prove (2-5). Let us only prove (5) as (2-4) are proved analogously. Equation (5) is equivalent to $$\sigma_R([\wh{f}\underset{i=1}{\overset{m_1}{\#_1}}a_i\omega_{1,i}])-\sigma_R([\wh{f}])=(0,2\K_1f\cap \alpha,W_1f\cap \alpha, 0)\in\wt{\Z^4},$$
which in turn follows by Calculation Lemma \ref{lem:calc}, cells (1, $\li$-$r_2$), and by Claim \ref{clm:delta}.
\end{proof}

\bigskip
\section{Proof of Surjectivity, Bijectivity, and Preimage Lemmas \ref{lem:surj}, \ref{lem:bij}, \ref{lem:pre}.}
Throughout this section we denote by $q_{k,i}$ the circle $S^1\times 0$ in the $i$-th connected component of $T_{m_k}$, see Fig.\ref{f:q}.
	\begin{figure}[H]
	\begin{center} 
	\includegraphics{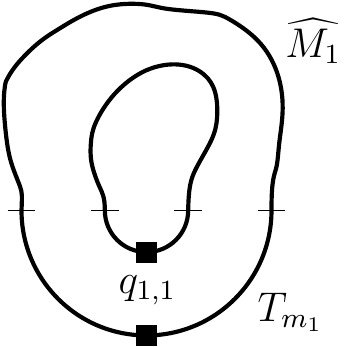} 
	\end{center} 
	\caption{The circle $q_{1,1}$ represented by a pair of squares.}
	\label{f:q} 
	\end{figure}
\smallskip
\subsection{Proof of Surjectivity Lemma \ref{lem:surj}.}
\begin{proof}[Proof of Surjectivity Lemma \ref{lem:surj}.]
We only prove that the map $\sigma$ is surjective. The map $\sigma_R$ is surjective by an analogous argument.

Choose an arbitrary embedding $f:M_1\sqcup M_2\rightarrow S^6$.
By general position, there are $2$-disks $\Delta_{k,i}$ in $S^6$ (see Fig.\ref{f:surj_lemma}), such that
\begin{itemize}
\item $\partial \Delta_{k,i}=f(q_{k,i})$,
\item interiors of all $\Delta_{k,i}$ are pairwise disjoint and are disjoint with $f(M_1\sqcup M_2)$. 
\end{itemize}

	\begin{figure}[H]
	\begin{center} 
	\includegraphics{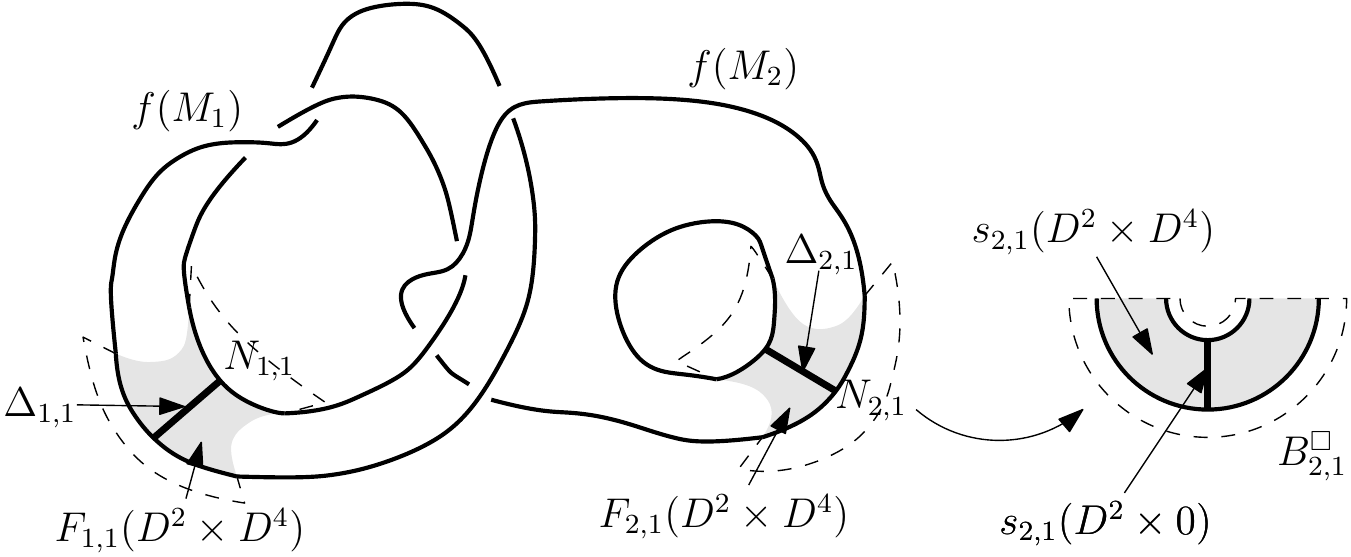} 
	\end{center} 
	\caption{Proof of Surjectivity Lemma \ref{lem:surj}.}
	\label{f:surj_lemma} 
	\end{figure}

The restriction of $f$ to the $i$-th component of $T_{m_k}$ can be extended to an embedding $F_{k,i}:D^2\times D^4\rightarrow S^6$ such that (see Fig.\ref{f:surj_lemma})
\begin{itemize}
\item $F_{k,i}(D^2\times 0)=\Delta_{k,i}$,
\item ${\rm Im}(F_{k,i})\cap {\rm Im}(f)=F_{k,i}(S^1\times D^2)$,
\item images of all the $F_{k,i}$ are pairwise disjoint.
\end{itemize}
Indeed, the obstruction to an extension to $D^2\times D^2$ is in $\pi_1(V_{4,2})=0$. Having $F_{k,i}$ already defined on $D^2\times D^2$ we can extend it to $D^2\times D^4$ without any obstructions.

Let $N_{k,i}$ be a tubular neighbourhood modulo $F_{k,i}(D^2\times S^3)$ of $F_{k,i}(D^2\times D^4)$ (see Fig.\ref{f:surj_lemma}). We can choose all $N_{k,i}$ to be pairwise disjoint and such that $N_{k,i}\cap f(\wh{M_1}\sqcup\wh{M_2})=F_{k,i}(S^1\times D^2)$. By construction, $F_{k,i}:D^2\times D^4\rightarrow N_{k,i}$ is isotopic to the composition of $s_{k,i}:D^2\times D^4\rightarrow B^{\square}_{k,i}$ with some diffeomorphism $B^{\square}_{k,i}\rightarrow N_{k,i}$, see the right part of Fig.\ref{f:surj_lemma}. There is a $6$-ball $B$ containing all of $N_{k,i}$, interior of $B$ being also disjoint with $f(\wh{M_1}\sqcup\wh{M_2})$.

Apply an ambient isotopy of $S^6$ which maps $B$ to $D^6_-$, each $N_{k,i}$ to $B^{\square}_{k,i}$, and each $F_{k,i}$ to $s_{k,i}$.

Denote by $f'$ the result obtained from  $f$ by the isotopy. By construction, $f'$ is in the image of $\sigma$.
\end{proof}

\smallskip
\subsection{Proof of the ``only if'' part of Preimage Lemma \ref{lem:pre}.}

We need the following Claim to prove the ``only if'' part of Preimage Lemma \ref{lem:pre}.
\begin{clm}
\label{clm:istd}
Let $[\wh{f}],[\wh{f'}]\in \wh{E}^6(\wh{M_1}\sqcup\wh{M_2})$ be isotopy classes. Suppose that embeddings $\sigma(\wh{f})$ and $\sigma(\wh{f'})$ are isotopic. Then there is a concordance between $\sigma(\wh{f})$ and $\sigma(\wh{f'})$ fixed on $T_{m_1}\sqcup T_{m_2}$.
\end{clm}
\begin{proof}
Denote $f:=\sigma(\wh{f})$ and $f':=\sigma(\wh{f'})$. By the definition of $\sigma$, we have that $f|_{T_{m_1}\sqcup T_{m_2}}=f'|_{T_{m_1}\sqcup T_{m_2}}=s_1\sqcup s_2$.

Clearly, it suffices to find a concordance between $f$ and $f'$ fixed on some tubular neighbourhood of each circle $q_{k,i}$.

Let $F:(M_1\sqcup M_2)\times I\rightarrow S^6$ be an isotopy between $f$ and $f'$. By general position, $F$ is isotopic relative to the boundary to some concordance $F'$ fixed on each $q_{k,i}$.

At each point of $F'(q_{1,1}\times I)$ identify with ${\mathbb R}^5$ the normal to $F'(q_{1,1}\times I)$ space in $S^6\times I$. The restriction of $F'$ to a small tubular neighbourhood of $q_{1,1}\times I$ gives us then a map $u:S^1\times I\rightarrow V_{5,2}$. We can choose the identification so that $u$ is constant on $S^1\times\partial I$. 

Let $\bar{u}:\frac{S^1\times I}{S^1\times\partial I}\rightarrow V_{5,2}$ be the quotient map. Space $\frac{S^1\times I}{S^1\times\partial I}$ is homotopically equivalent to $S^2\vee S^1$. From $\pi_2(V_{5,2})=\pi_1(V_{5,2})=0$ it follows that $\bar{u}$ is null-homotopic. Therefore $u$ is homotopic relative $S^1\times\partial I$ to the constant map.

It implies that isotopying $F'$ in a small tubular neighbourhood of $q_{1,1}\times I$ we can make $F'$ constant on this tubular neighbourhood. Doing this for all $k,i$ we get the required concordance.
\end{proof}

\begin{proof}[Proof of the ``only if'' part of Preimage Lemma \ref{lem:pre}.]
Suppose that $\sigma([\wh{f}])=\sigma([\wh{f}'])$ for some $[\wh{f}],[\wh{f}']\in \wh{E}^6(\wh{M_1}\sqcup\wh{M_2})$. Denote $f:=\sigma(\wh{f})$ and $f':=\sigma(\wh{f}')$.

By Claim \ref{clm:istd}, there is a concordance $F$ between $f$ and $f'$, fixed on $T_{m_1}\sqcup T_{m_2}$.

Denote $\Delta_{k,i}:=s_{k,i}(D^2\times 0)$. Disks $\Delta_{k,i}$ are pairwise disjoint, $\partial \Delta_{k,i}=f(q_{k,i})=f'(q_{k,i})=F_t(q_{k,i})$ for every $t\in I$, the interior of each $\Delta_{k,i}$ is disjoint with $s_1(T_{m_1})\sqcup s_2(T_{m_2})$.

For any $n$ and any two general position submanifolds $A,B\subset S^n$, ${\dim A}+{\dim B}=n$, denote by $\#|A\cap B|$ the algebraic number of points of intersection $A\cap B$. For each $\Delta_{k,i}$ denote by $\delet{\Delta_{k,i}}$ its interior.

For $1\leq i \leq m_1$, $1\leq j \leq m_2$ define
$$a_i:=\#|\delet{\Delta_{1,i}}\times I\cap F(M_1\times I)|,\text{\quad} b_i:=\#|\delet{\Delta_{1,i}}\times I\cap F(M_2\times I)|,$$ 
$$c_j:=\#|\delet{\Delta_{2,j}}\times I\cap F(M_2\times I)|,\text{\quad} d_j:=\#|\delet{\Delta_{2,j}}\times I\cap F(M_1\times I)|.$$

Denote
$$\wh{f''}:=\wh{f}\underset{i=1}{\overset{m_1}{\#_1}}a_i\omega_{1,i} \underset{i=1}{\overset{m_1}{\#_2}}b_i\omega_{i,1} \underset{j=1}{\overset{m_2}{\#_2}}c_j\omega_{2,j} \underset{j=1}{\overset{m_2}{\#_1}}d_j\omega_{2,j},$$
and
$$f'':=\sigma(\wh{f''}).$$

It remains to prove that $[\wh{f'}]=[\wh{f''}]$.

By construction, there is an isotopy $F''$ between $f$ and $f''$ which ``drags'' spheres $\omega_{1,i}$ and $\omega_{2,j}$ along pairwise disjoint embedded disks $s_{1,i}(0\times D^4)$ and $s_{2,j}(0\times D^4)$ for all $i$ and $j$. 
Isotopy $F''$ is fixed on $T_{m_1}\sqcup T_{m_2}$. We have that 
$$\#|\delet{\Delta_{1,i}}\times I\cap F''(M_1\times I)| = a_i,\text{\quad} \#|\delet{\Delta_{1,i}}\times I\cap F''(M_2\times I)| = b_i,$$
$$\#|\delet{\Delta_{2,j}}\times I\cap F''(M_2\times I)| = c_j,\text{\quad} \#|\delet{\Delta_{2,j}}\times I\cap F''(M_1\times I)| = d_j.$$

Consider now the concordance $H:=-F\cup F''$ between $f'$ and $f''$. By construction, $H$ is fixed on $T_{m_1}\sqcup T_{m_2}$ and 
$$\#|\delet{\Delta_{1,i}}\times I\cap H(M_1\times I)| = 0,\text{\quad} \#|\delet{\Delta_{1,i}}\times I\cap H(M_2\times I)| = 0,$$ 
$$\#|\delet{\Delta_{2,j}}\times I\cap H(M_2\times I)| = 0,\text{\quad} \#|\delet{\Delta_{2,j}}\times I\cap H(M_1\times I)| = 0.$$

So, using the Whitney trick (\cite[Theorem $6.6$]{Mi65}), we can isotope $H$, changing it only on $(\wh{M_1}\sqcup\wh{M_2})\times {\rm Int}I$, to some concordance $H'$ whose image is disjoint with each $\delet{\Delta_{1,i}}\times I$ and $\delet{\Delta_{2,j}}\times I$.

Now we can ``push'' the image of $H'$ away from each $\Delta_{1,i}\times I$ along the vectors of the normal framing of $\Delta_{1,i}\times I$ given by the embeddings $s_{1,i}(D^2\times D^4)$ (recall that $s_{k,i}(D^2\times 0)=\Delta_{k,i}$ by the definition of $\Delta_{k,i}$). Likewise we can ``push'' the image of $H'$ away from each $\Delta_{2,j}\times I$.

We obtain a new concordance $H''$ between $f'$ and $f''$ such that $H''((\wh{M_1}\sqcup\wh{M_2})\times I)\subset D^6_+\times I$. The restriction of $H''$ to $(\wh{M_1}\sqcup\wh{M_2})\times I$ is a concordance between $\wh{f'}$ and $\wh{f''}$ in $D^6_+$ fixed on the boundary. In codimension at least $3$ concordance implies isotopy, see \cite[Theorem $2.1$]{Hu70}, therefore $[\wh{f'}]=[\wh{f''}]$.
\end{proof}

\smallskip
\subsection{Proof of Bijectivity Lemma \ref{lem:bij}.}

To prove the Bijectivity Lemma \ref{lem:bij} we shall need the following analogue of Preimage Lemma \ref{lem:pre}.

For all $k\in\{1,2\}$ and $1\leq i \leq m_k$ define 
$$\omega_{R,k,i}:S^3\rightarrow S^6 \text{\quad by the formula \quad} \omega_{R,k,i}:=s_{R,k,i}|_{0\times S^3}.$$
 
\begin{lem}[Preimage']
\label{lem:pre'}
For any $[\wh{f}],[\wh{f}']\in \wh{E}^6(\wh{M_1}\sqcup\wh{M_2})$ we have that $\sigma_R([\wh{f}])=\sigma_R([\wh{f}'])$ if and only if
$$[\wh{f}']=[\wh{f}\underset{i=1}{\overset{m_1}{\#_1}}a_i\omega_{R,1,i} \underset{i=1}{\overset{m_1}{\#_2}}b_i\omega_{R,i,1} \underset{j=1}{\overset{m_2}{\#_2}}c_j\omega_{R,2,j} \underset{j=1}{\overset{m_2}{\#_1}}d_j\omega_{R,2,j}]$$ 
for some integers $a_i$, $b_i$, $c_j$, and $d_j$.
\end{lem}

The proof of Preimage' Lemma \ref{lem:pre'} is analogous to the proof of Preimage Lemma \ref{lem:pre}.

\begin{proof}[Proof of Bijectivity Lemma \ref{lem:bij}.]
Let us first prove that the restriction $\sigma_R|_{\wh{\WL}^{-1}(x)}$ is surjective.

Choose any $[g]\in E^6(S^3\sqcup S^3)$. The map $\wh{\WL}$ is surjective, which is proved analogously to the surjectivity of $\WL$ (part (I) of Theorem \ref{thm:main}). So, we can choose some $[\wh{f}]\in \wh{E}^6(\wh{M_1}\sqcup\wh{M_2})$ such that $\wh{\WL}(\wh{f})=x$.

There is an isotopy class $[g']\in E^6(S^3\sqcup S^3)$ such that $\sigma_R([\wh{f}])\#[g']=[g]$. Then $\sigma_R([\wh{f}]\#[g'])=[g]$ and $\wh{\WL}([\wh{f}]\#[g'])=\wh{\WL}([\wh{f}])=x$.

Let us now prove that the restriction $\sigma_R|_{\wh{\WL}^{-1}(x)}$ is injective. Let $[\wh{f}],[\wh{f}']\in \wh{E}^6(\wh{M_1}\sqcup\wh{M_2})$ be some isotopy classes such that $\sigma_R([\wh{f}])=\sigma_R([\wh{f}'])$ and $\wh{\WL}(\wh{f})=\wh{\WL}(\wh{f}')=x$.

By Preimage' Lemma \ref{lem:pre'}, we have that
$$[\wh{f}']=[\wh{f}\underset{i=1}{\overset{m_1}{\#_1}}a_i\omega_{R,1,i} \underset{i=1}{\overset{m_1}{\#_2}}b_i\omega_{R,i,1} \underset{j=1}{\overset{m_2}{\#_2}}c_j\omega_{R,2,j} \underset{j=1}{\overset{m_2}{\#_1}}d_j\omega_{R,2,j}]$$
for some integers $a_i$, $b_i$, $c_j$, and $d_j$.

Similarly to $\m$ and $\p_{k,i}$ denote
$$\p:=S^1\times *\subset S^1\times D^2$$
and
$$\m_{k,i} := P_{k,i}R\p.$$

By $[\m_{1,i}]$ we denote the corresponding homology classes in $\wh{M_1}$ (analogously to $[\p_{1,i}]$).
Let us compute $\lk(\underset{i=1}{\overset{m_1}{\Sigma}}a_i[\m_{1,i}],P_{1,1}q_{1,1})$ in two ways.

Circles $\m_{k,i}$ are meridians of the pairwise disjoint embedded solid tori $P_{k,i}(S^1\times D^2)\subset S^3$ and $P_{1,1}q_{1,1}$ is the parallel of the solid torus $P_{1,1}(S^1\times D^2)\subset S^3$. Therefore, $$\lk(\underset{i=1}{\overset{m_1}{\Sigma}}a_i[\m_{1,i}],P_{1,1}q_{1,1})=a_1.$$

By the analogue of Calculation Lemma \ref{lem:calc} (cell (1, $W_1$)), we have that 
$$\wh{W}_1(\wh{f'})=\wh{W}_1(\wh{f})+\underset{i=1}{\overset{m_1}{\Sigma}}a_i[\m_{1,i}].$$
Since $\wh{W}_1(\wh{f'})=\wh{W}_1(\wh{f})$, it follows that $\underset{i=1}{\overset{m_1}{\Sigma}}a_i[\m_{1,i}]=0\in H_1(\wh{M_1})$. Circle $P_{1,1}q_{1,1}\subset S^3$ is disjoint with $\wh{M_1}\subset S^3$, therefore $$\lk(\underset{i=1}{\overset{m_1}{\Sigma}}a_i[\m_{1,i}],P_{1,1}q_{1,1})=0.$$

Combining the last two paragraphs we get that $a_1=0$. By the same argument, $a_i=b_j=c_j=d_j=0$ for all $i,j$, meaning that $[\wh{f}]=[\wh{f'}]$.
\end{proof}

\bigskip
\section{Proof of Calculation Lemma \ref{lem:calc}.}
For the proof of Calculation Lemma \ref{lem:calc} we use Lemma \ref{lem:Lnew} which can be seen as an alternative definition of the linking coefficients $\li$ and $\lii$. We shall also need additional Claim \ref{clm:C}.

\smallskip
\subsection{Definition of framed intersections and preimages.}

Let $A,B\subset S^n$ be submanifodls in general position. Suppose that $B$ is framed. Then the {\it framed intersection} $A\cap B$ is a framed submanifold of $A$. The framing of $A\cap B\subset A$ is obtained by the projection of the framing of $B$ onto the tangent space of $A$ and subsequent Gram-Schmidt orthonormalising process.

Let $f:A\rightarrow S^n$ be an embedding and let $a\subset f(A)$ be a framed submanifold of $f(A)$. Then $f^{-1}(a)$ is called a {\it framed preimage} of $a$. The framing of $f^{-1}(a)$ is the $df^{-1}$-image of the framing of $a$.

Recall that $h$ denotes the Hopf invariant of a framed $1$-submanifold of $S^3$.
\begin{lem}
\label{lem:Lnew}
Let $g:S^3_1\sqcup S^3_2\rightarrow S^6$ be an embedding, where $S^3_1$ and $S^3_2$ are two distinct copies of $S^3$. Suppose that the restriction of $g$ to each connected component is trivial. Let $\Delta_1$, $\Delta_2$ be framed embedded disks in general position bounded by $gS^3_1$ and $gS^3_2$, respectively. Then
$$\li(g)=h (g^{-1}(gS^3_1\cap \Delta_2)) \text{\quad and\quad} \lii(g)=h (g^{-1}(gS^3_2\cap \Delta_1)).$$
\end{lem}
\begin{proof}
We only prove the first claim as the second one is analogous. Clearly, $\Delta_2$ is the preimage of a regular point of some homotopy equivalence $S^6\setminus gS^3_2\rightarrow S^2$. Therefore $gS^3_1\cap \Delta_2$ is the preimage of the same point under the restriction of this homotopy equivalence to $gS^3_1$. The first claim of the lemma now holds by the definition of $\li$.
\end{proof}

\smallskip
\subsection{Definition of $\Delta_{\omega,k,i}$.}

Let $\Delta_{\omega,k,i}\subset \partial D^6_-$ be an embedded framed $4$-disk bounded by $\omega_{k,i}(S^3)$ and such that for any $[\wh{f}]\in \wh{E}^6(\wh{M_1}\sqcup\wh{M_2})$
$$\wh{f}^{-1}(\wh{f}\wh{M_k}\cap\Delta_{\omega,k,i})=(\sigma_R\wh{f})^{-1}(\sigma_R\wh{f}(S^3_k)\cap\Delta_{\omega,k,i})=\p_{k,i},$$
where $S^3_k$ is the $k$-th component of the domain of $\sigma_Rf$, see Fig.\ref{f:omega}. Here the ``$=$'' signs mean the equality of both sides as framed submanifolds. The first equality holds by definition of $\sigma_R$ and the second equality is a part of the definition of $\Delta_{\omega,k,i}$.

\begin{clm}
\label{clm:C}
For any $k\in\{1,2\}$, $1\leq i\leq m_k$ there exist a disk $\Delta_{\omega,k,i}\subset \partial D^6_-$ satisfying the properties above.
\end{clm}
The claim is made obvious by Fig.\ref{f:omega}.

\smallskip
\subsection{Proof of Calculation Lemma \ref{lem:calc}.}

We shall prove the first two rows of the table, the proof for the second two rows is analogous.
Without a loss of generality we may assume that $i=1$. For brevity denote $\omega:=\omega_{1,1}$ and $\Delta_\omega:=\Delta_{\omega,1,1}$.

Without a loss of generality we may also assume that the restriction of $\sigma_R\wh{f}$ to the second component is trivial. Indeed, for any $g:S^3\rightarrow D^6_+$ whose image is far away from the images of $\wh{f}$ and $\wh{f'}$ we may substitute $\wh{f}$ and $\wh{f'}$ by $\wh{f}\#_2g$ and $\wh{f'}\#_2g$, respectively, without changing any of the table entries. By choosing $g$ appropriately we can always make the restriction of $\sigma_R\wh{f}$ to the second component trivial.

Let $F_k:S^3\rightarrow S^6$ be the restriction of $\sigma_Rf$ to the $k$-th component. Embedding $F_2$ is trivial by the argument in the previous paragraph.

	\begin{figure}[H] 
	\begin{center} 
	\includegraphics{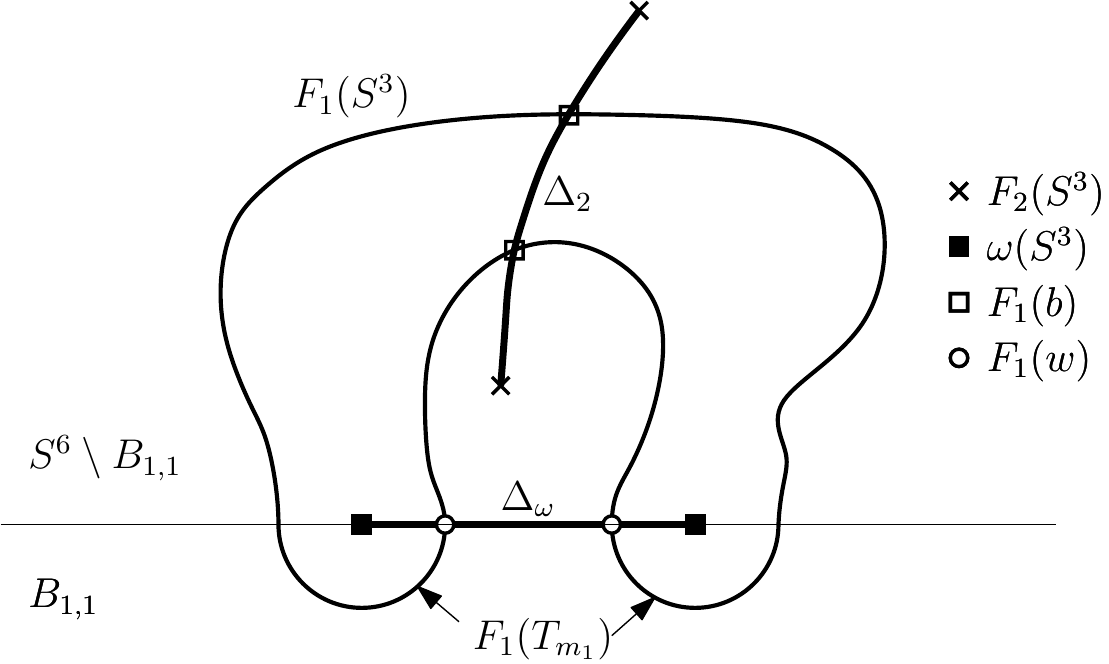} 
	\end{center} 
	\caption{Proof of Calculation Lemma \ref{lem:calc}.}
	\label{f:calc} 
	\end{figure}

Consider some embedded framed $4$-disk $\Delta_2$ in the complement to $B_{1,1}$ bounded by $F_2(S^3)$. Denote
$$w:=F_1^{-1}(F_1(S^3)\cap\Delta_{\omega}),$$
$$b:=F_1^{-1}(F_1(S^3)\cap\Delta_2).$$
Both $w$ and $b$ are framed $1$-submanifolds of $S^3$. Recall that $w=\p_{1,1}$ as a framed submanifold by Claim \ref{clm:C}.

We prove the second row of the table first.

{\bf Cell} (2, $\li$).
In this cell we need to compute $\li(\sigma_R(\wh{f}\#_2\omega_{1,1}))-\li(\sigma_R(\wh{f}))= \lx(F_1,F_2\#\omega)-\lx(F_1,F_2)$.

The disks $\Delta_2$ and $\Delta_\omega$ are disjoint by construction. So there is a framed embedded disk $\Delta_{F_2\#\omega}$, bounded by $(F_2\#\omega)(S^3)$ and such that $F_1S^3\cap \Delta_{F_2\#\omega}=(F_1S^3\cap \Delta_2)\sqcup (F_1S^3\cap \Delta_\omega)$. So by Lemma \ref{lem:Lnew} we have  
\begin{multline*}
\lx(F_1,F_2\#\omega)-\lx(F_1,F_2) = h(F_1^{-1}(F_1\cap \Delta_{F_2\#\omega}))-h(F_1^{-1}(F_1\cap \Delta_2)) = \\
= h(b\sqcup w) - h(b) = h(b) + h(w) + 2\lk(b,w) - h(b) = h(w) + 2\lk(b,w) =\\
= h(\p_{1,1}) + 2\lk(b,\p_{1,1}) = 0 + 2\lk(b,\p_{1,1}) = 2\lk(\wh{\K}_1\wh{f},\p_{1,1}).
\end{multline*}
The equation before the last holds because $h(\p_{1,1})=0$ by Claim \ref{clm:untwist}. The last equation holds by Claim \ref{clm:WZeifert} (take $\Delta_2\cap D^6_+$ as ``$\Delta$'' in the statement of the claim. Clearly, $\Delta_2\cap D^6_+$ satisfies the necessary condition by construction.).

{\bf Cell} (2, $\lii$).
In this cell we need to compute $\lii(\sigma_R(\wh{f}\#_2\omega_{1,1}))-\lii(\sigma_R(\wh{f}))= \lx(F_2\#\omega,F_1)-\lx(F_2,F_1)$.

We have
$$\lx(F_2\#\omega,F_1)-\lx(F_2,F_1)=\lx(F_2,F_1)+\lx(\omega,F_1)-\lx(F_2,F_1)=\lx(\omega,F_1)=\x_{1,1}(\wh{f}).$$
The first equation holds by Lemma \ref{lem:Lass}. The last equation is the definition of $\x_{1,1}(\wh{f})$.

{\bf Cell} (2, $r_1$).
In this cell we have $r_1(\sigma_R(\wh{f}\#_2\omega_{1,1}))-r_1(\sigma_R(\wh{f}))=0$ because the restrictions of $\sigma_R(\wh{f}\#_2\omega_{1,1})$ and $\sigma_R(\wh{f})$ to the first connected component are the same by the definition of $\#_2$.

{\bf Cell} (2, $r_2$).
By construction $\omega$ is trivial and the images of $\omega$ and $F_2$ lie in disjoint $6$-balls. So in this cell we have $r_2(\sigma_R(\wh{f}\#_2\omega_{1,1}))-r_2(\sigma_R(\wh{f}))=r(F_2\#\omega)-r(F_2)=0$.

{\bf Cells} (2,${\wh{W}_1}$-$ {\wh{\K}_2}$).
Clearly, there is a homotopy between $\wh{f}\#_2\omega$ and $\wh{f}$ which shrinks $\omega(S^3)$ along the disk $\Delta_\omega$. The disk $\Delta_\omega$ is disjoint with the image of $\wh{M_2}$ and the homotopy is the identity on $\wh{M_1}$. So $\wh{f'}=\wh{f}\#_2\omega$ and $\wh{f}$ differ at only one Whitney invariant out of four, namely
$$\wh{\K}_1(\wh{f}\#_2\omega)-\wh{\K}_1(\wh{f})=\wh{f}^{-1}[\wh{f}(\wh{M_1})\cap\Delta_\omega]=[w]=[\p_{1,1}].$$

{\bf Cell} (1, $\li$).
In this cell we need to compute $\li(\sigma_R(\wh{f}\#_1\omega_{1,1}))-\li(\sigma_R(\wh{f}))= \lx(F_1\#\omega,F_2)-\lx(F_1,F_2)$.

We have
$$\lx(F_1\#\omega,F_2) -\lx(F_1,F_2) = \lx(F_1,F_2) +\lx(\omega,F_2) -\lx(F_1,F_2) =\lx(\omega,F_2) = 0.$$
The first equation holds by Lemma \ref{lem:Lass}. The last equation holds because the images of $\omega$ and $F_2$ lie in disjoint $6$-balls.

{\bf Cell} (1, $\lii$).
In this cell we need to compute $\lii(\sigma_R(\wh{f}\#_1\omega_{1,1}))-\lii(\sigma_R(\wh{f}))= \lx(F_2,F_1\#\omega)-\lx(F_2,F_1)$.

By Lemma \ref{lem:conn} we have
$$2r(F_1\#F_2\#\omega) = \lx(F_2,F_1\#\omega)+,$$
$$2r(F_1\#F_2\#\omega) = \lx(F_2\#\omega,F_1) + \lx(F_1,F_2\#\omega) + 2r(F_2\#\omega)+2r(F_1).$$
So
$$\lx(F_2,F_1\#\omega) = \lx(F_2\#\omega,F_1) + \lx(F_1,F_2\#\omega) + 2r(F_2\#\omega)+2r(F_1) - \lx(F_1\#\omega,F_2) - 2r(F_2)-2r(F_1\#\omega).$$
Applying Lemma \ref{lem:Lass} and Lemma \ref{lem:conn} we get
\begin{multline*}
\lx(F_2,F_1\#\omega) = \lx(F_2\#\omega,F_1) + \lx(F_1,F_2\#\omega) + 2r(F_2\#\omega)+2r(F_1) - \lx(F_1\#\omega,F_2) - 2r(F_2)-2r(F_1\#\omega) = \\
= \lx(F_2,F_1) + \lx(\omega,F_1) + \lx(F_1,F_2\#\omega) + 2r(F_2)+2r(\omega)+\lx(F_2,\omega) +\lx(\omega,F_2)+2r(F_1)-\\
-\lx(F_1,F_2)-\lx(\omega,F_2)-2r(F_2)-2r(F_1)-2r(\omega) - \lx(F_1,\omega)-\lx(\omega,F_1) =\\
= \lx(F_2,F_1) + \lx(F_1,F_2\#\omega) +\lx(F_2,\omega) -\lx(F_1,F_2) - \lx(F_1,\omega) = \\ 
= \lx(F_2,F_1) + \lx(F_1,F_2\#\omega) -\lx(F_1,F_2) - \lx(F_1,\omega),
\end{multline*}
where the last equation holds because $\lx(F_2,\omega)=0$ (see paragraph ``Cell (1, $\li$)'').
So
\begin{multline*}
\lx(F_2,F_1\#\omega)-\lx(F_2,F_1) = \lx(F_2,F_1) + \lx(F_1,F_2\#\omega) -\lx(F_1,F_2) - \lx(F_1,\omega) - \lx(F_2,F_1) = \\ = \lx(F_1,F_2\#\omega) -\lx(F_1,F_2) - \lx(F_1,\omega).
\end{multline*}
From the paragraph ``Cell (2, $\li$)'' we know that $\lx(F_1,F_2\#\omega) -\lx(F_1,F_2)= 2\lk(\wh{\K}_1\wh{f},\p_{1,1})$. Also, by Lemma \ref{lem:Lnew}, $\lx(F_1,\omega)=h(w)=h(\p_{1,1})$ and by Claim \ref{clm:untwist}, $h(\p_{1,1})=0$, so $\lx(F_1,\omega)=0$. We get
$$\lx(F_2,F_1\#\omega)-\lx(F_2,F_1) = 2\lk(\wh{\K}_1\wh{f},\p_{1,1}).$$

{\bf Cell} (1, $r_1$).
In this cell we need to compute $r_1(\sigma_R(\wh{f}\#_1\omega_{1,1}))-r_1(\sigma_R(\wh{f})) = r(F_1\#\omega)-r(F_1)$. Applying Lemma \ref{lem:conn} we get
\begin{multline*}
r(F_1\#\omega)-r(F_1)=r(F_1) + r(\omega) + \frac{\lx(F_1,\omega)+\lx(\omega,F_1)}{2} -r(F_1)=\\
 = r(\omega) + \frac{\lx(F_1,\omega)+\lx(\omega,F_1)}{2}.
\end{multline*}
We know that $r(\omega)=0$ because $\omega$ is trivial. Also, $\lx(F_1,\omega)=0$, see the end of paragraph ``Cell (1, $\lii$)''. So
$$r(F_1\#\omega)-r(F_1)=\frac{\lx(\omega,F_1)}{2}=\frac{\x_{1,1}(\wh{f})}{2}$$
by the definition of $\x_{1,1}$.

{\bf Cell} (1, $r_2$).
Analogous to cell (2, $r_1$).

{\bf Cells} (1,$\wh{W}_1$-$\wh{\K}_2$).
Analogous to cells (2,$\wh{W}_1$-$\wh{\K}_2$).

\bigskip
\section{Proof of Claim \ref{clm:Wpre} and Linking Lemma \ref{lem:Linking}.}

\smallskip
\subsection{Proof of Claim \ref{clm:Wpre}.}

By Surjectivity Lemma \ref{lem:surj}, there are $\sigma$-preimages $[\wh{f}]$ and $[\wh{f'}]$ of $[f]$ and $[f']$, respectively.
The group $H_1(M_1)$ is obtained from $H_1(\wh{M_1})$ by adding the relation $[\p_{1,i}]=0$ for each $1\leq i\leq m_1$. Since $W_1(f)=W_1(f')$, it follows that $\wh{W}_1(\wh{f'})-\wh{W}_1(\wh{f})=\overset{m_1}{\underset{i=1}{\Sigma}}a_i[\p_{1,i}]$ for some integers $a_i$. 

Redefine $\wh{f}:=\wh{f}\underset{i=1}{\overset{m_1}{\#_1}}a_i\omega_{1,i}$. By Preimage Lemma \ref{lem:pre}, we still have $\sigma(\wh{f})=[f]$. By Calculation Lemma \ref{lem:calc}, we now have $\wh{W}_1(\wh{f})=\wh{W}_1(\wh{f'})$. Performing the analogous operation for the remaining three invariants $\wh{\K}_1$, $\wh{W}_2$, and $\wh{\K}_2$, we can achieve that $\wh{\WL}(\wh{f})=\wh{\WL}(\wh{f'})$. Then $[\wh{f}]$ and $[\wh{f'}]$ are as required.

\smallskip
\subsection{Proof of Linking Lemma \ref{lem:Linking}.}

To prove Linking Lemma \ref{lem:Linking} we shall need the following claim and lemma.

\begin{clm}
\label{clm:xy}
Let $[\wh{f}], [\wh{f'}]\in \wh{E}^6(\wh{M_1}\sqcup\wh{M_2})$ be isotopy classes. Then there are embeddings $g_1:S^3\rightarrow {\rm Int}D^6_+$, $g_2:S^3\rightarrow {\rm Int}D^6_+$, and $g:S^3\sqcup S^3\rightarrow {\rm Int}D^6_+$ such that
\begin{itemize}
\item isotopy classes $[g_1]$ and $[g_2]$ are trivial,
\item images of $g_1$ and $g_2$ are pairwise disjoint and disjoint with the image of $\wh{f}$,
\item image of $g$ lie in a $6$-ball disjoint with the images of $\wh{f}$, $g_1$, and $g_2$,
\item $[\wh{f}\#_1g_1\#_2g_2]\#[g]=[\wh{f'}]$.
\end{itemize}
In the special case $\wh{\WL}(f)=\wh{\WL}(f')$ we may choose $g$ so that a simpler equation
$$[\wh{f}]\#[g]=[\wh{f'}]$$
holds.
\end{clm}
\begin{proof}
The special case of the claim is proved analogously to part {\bf (II)} of Theorem \ref{thm:main}.
Consider the general case. Analogously to the proof of part {\bf (I)} of Theorem \ref{thm:main} we may choose $g_1,g_2:S^3\rightarrow {\rm Int}D^6_+$ so that $\wh{\WL}(\wh{f}\#_1g_1\#_2g_2)=\wh{\WL}(f')$. Now apply the special case of the claim to isotopy classes $[\wh{f}\#_1g_1\#_2g_2]$ and $[f']$.
\end{proof}

\begin{lem}
\label{lem:Linking'}
For any $k\in\{1,2\}$, $1\leq i\leq m_k$, and $[\wh{f}],[\wh{f'}]\in \wh{E}^6(\wh{M_1}\sqcup\wh{M_2})$ the following equality holds
$$\x_{k,i}(\wh{f'})-\x_{k,i}(\wh{f})=2{\rm lk}(\p_{k,i},\wh{W}_k(\wh{f'})-\wh{W}_k(\wh{f})).$$
\end{lem}
\begin{proof}
Let $g_1:S^3\rightarrow {\rm Int}D^6_+$, $g_2:S^3\rightarrow {\rm Int}D^6_+$, and $g:S^3\sqcup S^3\rightarrow {\rm Int}D^6_+$ be embeddings as in the statement of Claim \ref{clm:xy}. 
Denote 
\begin{itemize}
\item by $F$, $F'$, and $G'$ the restrictions of $\sigma_R(\wh{f})$, $\sigma_R(\wh{f'})$, and $g$ to the $k$-th component, respectively,
\item $G:=g_k$,
\item $\omega:=\omega_{k,i}$.
\end{itemize}
By Claim \ref{clm:xy} we have $[F']=[F\#G\#G']$. The isotopy between $F'$ and $F\#G\#G'$ is fixed on $\omega(S^3)\subset D^6_-$ so without a loss of generality we may assume that $F'=F\#G\#G'$.

By the definition of $\x_{k,i}$ we have
\begin{equation}
\label{eq1:lem:Linking1}
\x_{k,i}(\wh{f'})-\x_{k,i}(\wh{f})= \lx(\omega,F\#G\#G')-\lx(\omega,F) = \lx(\omega,F\#G)-\lx(\omega,F)
\end{equation}
where the last equality holds because the image of $G'$ lies in a $6$-ball in $D^6_+$ disjoint from the images of $\omega$, $F$, and $G$.

Let us compute $\lx(\omega, F\#G)$. Next two equalities follow from Lemma \ref{lem:conn}
$$2r(\omega\#F\#G) = \lx(\omega, F\#G) + \lx(F\#G, \omega) + 2r(\omega) + 2r(F\#G),$$
$$2r(\omega\#F\#G) = \lx(F,\omega\#G)+\lx(\omega\#G,F)+2r(\omega\#G)+2r(F).$$
We get
$$\lx(\omega, F\#G) = \lx(F,\omega\#G)+\lx(\omega\#G,F)+2r(\omega\#G)+2r(F)-\lx(F\#G, \omega) - 2r(\omega) - 2r(F\#G).$$
Clearly, $\omega$ is trivial so $r(\omega)=0$. Also, $G$ is trivial by Claim \ref{clm:xy}. Moreover, image of $\omega$ is in the boundary of $D^6_+$ while the image of $G$ is in the interior of $D^6_+$. So, $r(\omega\#G)=0$. Now we can simplify the formula for $\lx(\omega, F\#G)$ above to get
$$\lx(\omega, F\#G) = \lx(F,\omega\#G)+\lx(\omega\#G,F)+2r(F)-\lx(F\#G, \omega) - 2r(F\#G).$$
By Lemma \ref{lem:conn}, we have $2r(F\#G)=2r(F)+2r(G)+\lx(F,G)+\lx(G,F)=2r(F)+\lx(F,G)+\lx(G,F)$. So
$$\lx(\omega, F\#G) = \lx(F,\omega\#G)+\lx(\omega\#G,F)-\lx(F\#G, \omega) - \lx(F,G)-\lx(G,F).$$
By Lemma \ref{lem:Lass}, we have $\lx(\omega\#G,F)=\lx(\omega,F)+\lx(G,F)$ and $\lx(F\#G, \omega)=\lx(F, \omega)+\lx(G, \omega)=\lx(F, \omega)$, where the last equality holds because the images of $\omega$ and $G$ lie in disjoint $6$-balls meaning that $\lx(G, \omega)=0$. So
$$\lx(\omega, F\#G) = \lx(F,\omega\#G)+\lx(\omega,F)-\lx(F, \omega) - \lx(F,G).$$
Going back to equation (\ref{eq1:lem:Linking1}) we get
$$\x_{k,i}(\wh{f'})-\x_{k,i}(\wh{f})=\lx(F,\omega\#G)-\lx(F, \omega) - \lx(F,G).$$

Let $\Delta_G\subset{\rm Int}D^6_+$ be an embedded framed disk bounded by $G(S^3)$.
Denote 
$$d:=F^{-1}(F(S^3)\cap\Delta_G)$$
and
$$w:=F^{-1}(F(S^3)\cap\Delta_{\omega,k,i}).$$

By Lemma \ref{lem:Lnew} we have
\begin{itemize}
\item $\lx(F,\omega\#G)=h(w\sqcup d)$,
\item $\lx(F, \omega)=h(w),$
\item $\lx(F,G)=h(d).$
\end{itemize}

So
\begin{multline*}
\x_{k,i}(\wh{f'})-\x_{k,i}(\wh{f})= h(w\sqcup d) - h(w) - h(d) = \\
= h(w)+h(d)+2{\rm lk}(w,d)-h(w)-h(d) = 2{\rm lk}(w,d) = 2{\rm lk}(\p_{k,i},\wh{W}_k(\wh{f'})-\wh{W}_k(\wh{f})).
\end{multline*}

The last equation holds because
\begin{itemize}
\item $w=\p_{k,i}$ by Claim \ref{clm:C},
\item $d\subset {\rm Int}M_k$ is a representative of the homology class $\wh{W}_k(\wh{f'})-\wh{W}_k(\wh{f})\in H_1(\wh{M_k})$ by the definition of $\wh{W}_k$.
\end{itemize}  
\end{proof}

\begin{proof}[Proof of Linking Lemma \ref{lem:Linking}]
By Lemma \ref{lem:Linking'}, it is enough to prove the lemma in the special case $\wh{f}=\wh{f}^0$. I.e., we need to prove that
$$\overset{m_k}{\underset{i=1}{\Sigma}}a_i[\p_{k,i}]=0 \text{\quad}\Rightarrow\text{\quad} \overset{m_k}{\underset{i=1}{\Sigma}}a_i\x_{k,i}(\wh{f}^0) = \overset{m_k}{\underset{i=1}{\Sigma}}2\lk(\p_{k,i}, \wh{W}_k(\wh{f}^0)).$$
The righthand side is zero because $\wh{W}_k(\wh{f}^0)=0$ by definition. Therefore we need to prove that
$$\overset{m_k}{\underset{i=1}{\Sigma}}a_i[\p_{k,i}]=0 \text{\quad}\Rightarrow\text{\quad} \overset{m_k}{\underset{i=1}{\Sigma}}a_i\x_{k,i}(\wh{f}^0) = 0.$$

Consider the embedding $\wh{f'}:=\wh{f}^0\overset{m_k}{\underset{i=1}{\#_k}}a_i\omega_{k,i}$.

We have that  
$$\wh{W}_k(\wh{f'})=\wh{W}_k(\wh{f'})-\wh{W}_k(\wh{f}^0)=\overset{m_k}{\underset{i=1}{\Sigma}}a_i[\p_{k,i}]=0,$$
where the first equation holds because $\wh{W}_k(\wh{f}^0)=0$ and the second equation holds by Calculation Lemma \ref{lem:calc}. Also by Calculation Lemma \ref{lem:calc}, we get that the rest of the Whitney invariants of $\wh{f'}$ and $\wh{f}^0$ are also the same, namely $\wh{WL}(\wh{f'})=\wh{WL}(\wh{f}^0)=0$. 

By Claim \ref{clm:xy} (the ``special case''), there is an embedding $g:S^3\sqcup S^3\rightarrow S^6$ such that 
\begin{equation}
\label{eq:Linking_1}
[\wh{f}^0]\#[g]=[\wh{f'}].
\end{equation}
On one hand, from the commutativity of the action $\#$ (Claim \ref{clm:comm}) we get
$$r_k(\sigma_R \wh{f}')-r_k(\sigma_R\wh{f}^0)=r_k(g).$$
On the other hand, by Calculation Lemma \ref{lem:calc}, we get
$$r_k(\sigma_R \wh{f}')-r_k(\sigma_R\wh{f}^0)=\overset{m_k}{\underset{i=1}{\Sigma}}a_i\frac{\x_{k,i}(\wh{f}^0)}{2}.$$
So
$$\overset{m_k}{\underset{i=1}{\Sigma}}a_i\frac{\x_{k,i}(\wh{f}^0)}{2}=r_k(g).$$
It remains to prove that $r_k(g)=0$.

By the commutativity of the action $\#$, we get from (\ref{eq:Linking_1}) that
$$\sigma([\wh{f}'])=\sigma([\wh{f}^0])\#[g].$$
On the other hand, by Preimage Lemma \ref{lem:pre}, we have
$$\sigma([\wh{f}'])=\sigma([\wh{f}^0]).$$
Therefore, 
$$\sigma([\wh{f}^0])\#[g]=\sigma([\wh{f}^0]).$$

Consider the restriction of $\sigma([\wh{f}^0])$ to $M_k$. Its Whitney invariant $W$ is equal to $W_k(\sigma \wh{f}^0)=W_k({f^0})=0$. So $r_k(g)=0$ by Theorem \ref{thm:Sk}, part {\bf (III)}.
\end{proof}

\end{document}